\documentclass[leqno,10pt]{amsart}
\usepackage{amssymb}
\usepackage{amssymb, amsmath}
\usepackage{color}
\usepackage{verbatim}

\allowdisplaybreaks

\let\eps=\varepsilon

\def\cI{{\mathcal I}}

\def\cT{{\mathcal T}}

\def\R{{\mathbb R}}

\def\Z{{\mathbb Z}}

\def\bS{{\mathbb S}}

\def\virgp{\raise 2pt\hbox{,}}
\def\cdotpv{\raise 2pt\hbox{;}}

\def\div{ \hbox{\rm div}\,  }
\def\curl{ \hbox{\rm curl}\,  }

\def\ddj{\dot \Delta_j}

\def\na{\nabla}
\def\lam{\lambda}

\def\pa{\partial}

\def\f{\frac}

\def\eqdefa{\buildrel\hbox{\footnotesize def}\over =}

\newtheorem{thm}{Theorem}[section]
\newtheorem{lem}{Lemma}[section]
\newtheorem{rmk}{Remark}[section]
\newtheorem{col}{Corollary}[section]
\newtheorem{prop}{Proposition}[section]

\newcommand{\ben}{\begin{eqnarray}}
\newcommand{\een}{\end{eqnarray}}
\newcommand{\beno}{\begin{eqnarray*}}
\newcommand{\eeno}{\end{eqnarray*}}

\begin{document}
\title[Stability of large solutions for full Navier-Stokes system]
{Stability of large solutions for full compressible Navier-Stokes equations in the whole spaces}

\author[L. He]{Lingbing He}
\address[L. He]{Department of Mathematical Sciences, Tsinghua University\\
Beijing 100084,  P. R.  China.} \email{lbhe@math.tsinghua.edu.cn}
\author[J. Huang]{Jingchi Huang}
\address[J. Huang]{ School of Mathematics, Sun Yat-sen University, Guangzhou Guangdong 510275, P. R. China. } \email{huangjch25@mail.sysu.edu.cn}
 \author[C. Wang]{Chao Wang}
\address[C. Wang]{ School of Mathematical Sciences, Peking University\\
 Beijing 100871,  P. R.  China.} \email{wangchao@math.pku.edu.cn }

\begin{abstract}
The current paper is devoted to the investigation of the global-in-time stability of large solutions for the full Navier-Stokes-Fourier system in the whole space. Suppose that the density and the temperature are bounded from above uniformly in time in the Holder space $C^\alpha$ with $\alpha$ sufficiently small and in $L^\infty$ space respectively. Then we prove two results:
(1). Such kind of the solution  will converge to its associated equilibrium with a rate which is the same as that for the heat equation if we impose the same condition on the initial data. As a result, we obtain the propagation of  positive lower bounds of the density and the temperature.
(2). Such kind of the solution is stable, that is,  any perturbed solution will remain close to the reference solution if initially they are close to each other. This shows that the set of the smooth and bounded solutions is open.

\end{abstract}
 \maketitle

\section{introduction}
The motion of the compressible viscous and heat-conductive gases is governed by the full Navier-Stokes-Fourier system:
\begin{equation}\label{cns}
 \left\{\begin{array}{l}
\partial_t\rho+\div(\rho  u )=0,\\[0.5ex]\displaystyle
\partial_t(\rho  u)+\div (\rho  u \otimes u )
+\nabla P(\rho,T)=\div \bS (u),\\
\pa_t(\rho E)+\div(\rho Eu)+\div(P(\rho,T)u)=\div\bigl(\bS(u)u\bigr)+\Delta T,\\
 \lim\limits_{|x|\rightarrow \infty} (\rho,u,T)=(1,0,1),
\end{array}
\right.
\end{equation}
where $(\rho,u,T)$ are the  density, the velocity field and the absolute temperature of the fluid, respectively. In the present work, we only consider the perfect heat conducting and viscous gases. In this case, the pressure $P(\rho, T)$ is given by $P=\rho T.$ $E$ denotes total energy, where $E=T+\f{|u|^2}2.$ Here $\bS(u)$ is the stress tensor, given by
$$\bS(u)=\mu\bigl(\na u+(\na u)'\bigr)+\lambda\div u \cI_3,$$x`
where $\cI_3$ is a $3\times 3$ unit matrix, and $A'$ means the transpose of matrix $A.$ $\mu$ and $\lambda$ are the coefficients of viscosity, which are assumed to be constants, satisfying the following physical restrictions:
$\mu>0$ and $2\mu+3\lambda\geq0.$

If the solutions are regular enough (such as strong solutions), \eqref{cns} is equivalence to the following system which is very useful in the proofs of the main theorems:
\begin{equation}\label{fcns}
 \left\{\begin{array}{l}
\partial_t\rho+\div(\rho  u )=0,\\[0.5ex]\displaystyle
\partial_t(\rho  u)+\div (\rho  u \otimes u )
+\nabla P=\mu\Delta u+(\mu+\lambda)\na\div u,\\
\pa_t(\rho T)+\div(\rho Tu)+\rho T\div u=\f{\mu}2|\na u+(\na u)'|^2+\lambda(\div u)^2+\Delta T.\\
 \end{array}
\right.
\end{equation}

\smallskip
In the present work, we are interested in the following two problems for the system \eqref{cns}:
(i). What is the long-time behavior of the solution   to   \eqref{cns} ?
(ii). Which kind of the solution to  \eqref{cns} is stable?

Obviously these two problems are fundamental for \eqref{cns}. However both of them are not solved well. The main obstruction comes from the existence of global smooth solution. So far, all the  results are restricted to the perturbation framework. In other words, the global solution is constructed near the equilibrium.  We refer readers to  \cite{MN1, MN2,Huang2} for the results in Sobolev space and to \cite{Dan-ARMA} in Critical Besov spaces.
Because of this restriction, the method on the global dynamics and the stability of \eqref{cns}  relies heavily on the analysis of the linearization of the system. We refer readers to  \cite{DX,Kobayashi1,Kobayashi2,Kobayashi3,Kobayashi4} and reference therein for details.  These results can be summarized as follows. Assume that the initial data $(\rho_0, u_0,T_0)$ is a small perturbation of equilibrium $(1, 0,1)$ in $L^1(\R^3)\times H^3(\R^3)\times H^3(\R^3).$ Then we can prove that \ben\label{decay}
\|(\rho-1, u,T-1)(t)\|_{L^2}\leq C(1+t) ^{-\f34}.
\een
This shows that in the close-to-equilibrium setting the   rate of the convergence of  the solution is the same as that for the heat equations if we put the same condition on the initial data. In this sense, \eqref{decay} can be regarded as  the optimal decay estimate for system \eqref{cns}. Finally we mention the work \cite{ZZ} on convergence to equilibruim for the full Navier-Stokes-Fourier system in the bounded domain.
\smallskip

The aim of the paper is to investigate the long-time behavior and the global-in-time stability of the solution to \eqref{cns} for the general initial data. To do that, we need to impose some assumptions on the solution itself which looks unsatisfactory. But our results are still interesting in the following sense:
\smallskip

(1). From the point of view of global dynamics, we disclose the stabilized mechanism for the system \eqref{cns} without the restriction that the solution is in the close-to-equilibrium setting. As a result, we obtain the long-time behavior of the solution   and the propagation of the lower bounds of the density and the temperature.

\smallskip
(2). Technically we introduce a new method to deal with the parabolic-hyperbolic system to catch the full dissipation. Compared to the standard linearization method, the new ingredients come from the basic energy identity and the coupling effect behind the system. The key observation is as follows. The basic energy identity shows that the system has the dissipation structure which is not complete. In the case of full Navier-Stokes-Fourier system, there is no dissipation for the density. However the coupling effect behind system help us to obtain the dissipation for the density. Finally the system will look like a heat equation. By the key observation that there exists some cancellation structure which helps us to get the control for the low frequency part of the solution, time-frequency splitting method can be applied to obtain the global dynamics: the propagation of the smoothness and the convergence to the equilibrium with the rate as the same as the result obtained by the linearization method.

\smallskip

To explain our strategy clearly, let us give a short review on \cite{HHW} where we work on the isentropic case. The proof relies on the uniform-in-time bounds for the propagation of the regularity and the dissipation inequality for the system. Roughly speaking, starting from the basic energy identity and the assumptions on the solution itself, in the first step, we want to get the uniform-in-time bounds for the high regularity for the solution. Then in the second step, by using the result built in the first step, we derive   the dissipation inequality for the system which helps to consider the longtime behavior of the solution and the propagation of the lower bounds for the density. Thanks to the interplay between these two steps, finally we obtain the global dynamics for the compressible Navier-Stokes equations.

\smallskip

Inspired by the isentropic case, to get the stability of full Navier-Stokes system, the first step is to get the uniform-in-time bounds for the propagation of the regularity. Because of thermal energy equation, we need to involve some new methods which come from the corresponding blow-up results (see \cite{SWZ2,wenzhu}). Firstly, because of the different definition of the effective viscosity flux $G$ which contains temperature now, some new terms from the $\eqref{fcns}_3$ come out(see Lemma \ref{energylem4}). To estimate them, we not only need to copy some argument from  isentropic case, but also need to use the coupling effect of the system. Secondly, the linearization of $\eqref{fcns}_3$ also brings some difficulties which need new idea to deal with. The biggest enemy comes from the term $\rho T\div u.$ Since the density and the temperature are not integrable, we have to rewrite the system  $\eqref{fcns}$ around the equilibrium $(1,0,1)$. Then we derive that $\rho T\div u=(\rho T-1)\div u +\div u$. Obviously the linear term $\div u$ will bring the trouble to obtain the uniform-in-time estimates for the solution. To deal with all the terms involved with this linear term in energy estimates, our key idea is making full use of effective viscosity flux, the cancellation of the system  and also the basic energy identity(see Lemma \ref{energylem3}-\ref{energylem5} in the after).




\bigskip

  Before we state our results, let us introduce the notations which are used throughout the paper. We use the notation $a\sim b$ whenever $a\le C_1 b$ and $b\le C_2
a$ where $C_1$ and $C_2$ are universal constants.  We denote $C(\lambda_1,\lambda_2,\cdots, \lambda_n)$ by a constant depending on parameters $\lambda_1,\lambda_2,\cdots, \lambda_n$.

\smallbreak

  Now we are in a position to state our main results on the system \eqref{cns}. Our first result is concerned with  the global dynamics of the equation \eqref{cns}.
\begin{thm}\label{thm:main1}
Let $\mu>\frac12 \lambda,$ and $(\rho,u,T)$ be a global and smooth  solution of \eqref{cns} with initial data $(\rho_0,u_0,T_0)$ where $\rho_0\geq c>0,$ $T_0\geq c>0,$ and satisfying the admissible condition  holds:
\begin{equation}\label{admissible}
\begin{split}
&u_t|_{t=0}=-u_0\cdot\na u_0+\f1{\rho_0}\div\bigl(\bS(u_0)\bigr)-\f1{\rho_0}\na (\rho_0T_0),\\
&T_t|_{t=0}=-u_0\cdot\na T_0-T_0\div u_0+\f1{\rho_0}\bS(u_0):\na u_0+\f1{\rho_0}\Delta T_0.
\end{split}
\end{equation}
Assume that $(\rho, T)$ satisfies that
\ben\label{Ass1}
&&\| \rho,  T\|_{L^\infty} \leq M_1,\label{Ass1}\\
&&\|\rho\|_{C^\alpha}\leq M_2,\label{Ass2}
\een
where  $\alpha$ is a positive number. Denote that $a\eqdefa \rho-1,$ $\theta\eqdefa T-1,$ then if $(a_0, u_0)\in L^1(\R^3)\cap H^2(\R^3)$ and $\theta_0\in L^1(\R^3)\cap H^1(\R^3),$ we have
\begin{enumerate}
\item{(Propagation of the lower bounds of the density and the temperature)} There exist two constants $c_1$ and $c_2$ depending on $c$ and $M$ such that
\beno \rho(t,x)\ge c_1>0,\quad T(t,x)\ge c_2>0. \eeno
\item{(Uniform-in-time bounds for the regularity)}
\beno \|a,u\|_{L^\infty_t H^2}^2+\|\theta\|_{L^\infty_t H^1}^2+\int_0^\infty \|a,\na u\|_{ H^2}^2+\|\na\theta\|_{ H^1}^2\,d\tau\le C(M, \|a_0\|_{L^1\cap H^2},\|u_0\|_{L^1\cap H^2},\|\theta_0\|_{L^1\cap H^1} ),\eeno
\item{(Longtime behavior of the solution)}
\begin{equation}\label{decayest}
\|u(t)\|_{H^1} + \|a(t)\|_{H^1}+ \|\theta(t)\|_{H^1} \leq C( M, \|a_0\|_{L^1\cap H^1},\|u_0\|_{L^1\cap H^2},\|\theta_0\|_{L^1\cap H^1} )(1+t)^{-\f34},\end{equation}
where $M=\max\{ M_1, M_2\}$.
\end{enumerate}
 \end{thm}

\bigskip

\begin{rmk} It shows that the upper bounds of the density and the temperature can control the propagation of   positive lower bounds of the density and the temperature.  \end{rmk}

\begin{rmk}    \eqref{decayest}  shows that our decay estimate  is comparable to  \eqref{decay} obtained in the perturbation framework. In fact, if we assume that the initial data is in $L^p$ with $p\in[0,1]$, then the decay estimate will turn to be
$$\|u(t)\|_{H^1} + \|a(t)\|_{H^1}+ \|\theta(t)\|_{H^1} \leq C( M, \|a_0\|_{L^1\cap H^1},\|u_0\|_{L^1\cap H^2},\|\theta_0\|_{L^1\cap H^1} )(1+t)^{-\f34(\f{2}{p}-1)},$$
which is exactly as the same as that for the heat equation.
\end{rmk}	

\begin{rmk}  \eqref{Ass1} and \eqref{Ass2} are always satisfied in the perturbation framework. This shows that our assumptions are reasonable.
\end{rmk}

 \bigskip

Next we are in the position to state our global-in-time stability result for the system \eqref{cns}.
\begin{thm}\label{thm:main2}
Let $(\bar\rho, \bar u,\bar T)$ be a global and smooth solution for the system \eqref{cns} with the initial data $(\bar \rho_0, \bar u_0,\bar T_0)$ verifying that   \ben\label{assm}
  \|\f{1}{\bar \rho}, \bar \rho, \na \bar\rho \|_{ L^\infty_t( H^s) }+\|\bar u, \bar u_t,\bar T, \bar {T}_t\|_{   L^\infty_t(  H^s)}+\|\na \bar u,\na \bar T\|_{  L^\infty_t(  H^s)}\leq C,
  \een
  where $s>\f32$. Assume that $(\bar\rho_0-1, \bar u_0, \bar{T}_0-1)\in L^1(\R^3)\cap H^s(\R^3)$ satisfying the assumptions of Theorem \ref{thm:main1}. There exists a $\eps_0=\eps_0(C)$ depending only on $C$ such that for any $0<\eps\leq\eps_0$, if
\ben\label{assm1}
\|\rho_0-\bar{\rho}_0\|  _{ H^s }+\|  u_0-\bar{u}_0 \|_{   H^s } +\|T_0-\bar{T}_0\|_{   H^s }  \leq \eps,
\een
then \eqref{cns} admits a global  and unique solution $(\rho, u,T)$ with the initial data $(\rho_0, u_0,T_0).$ Moreover, for any $t>0$,
\beno
\|\rho-\bar{\rho}\|  _{ H^s }+\|  u-\bar{u} \|_{   H^s } +\|T-\bar{T}\|_{   H^s }  \lesssim \min\{(1+\delta|\ln\eps|)^{-\f34}, (1+t)^{-\f34}+\epsilon\},  \eeno
where  $\delta$ is a constant independent of $\eps.$ \end{thm}


\setcounter{equation}{0}
\section{Global dynamics of the   Navier-Stokes-Fourier equations}
In this section, we give the proof to Theorem \ref{thm:main1}. To do that, we split our proof into two steps. In the first step, we want to obtain the uniform-in-time bounds for the propagation of the regularity. Then we use it to derive the
 dissipation inequality first. In the second step, we shall use the time-frequency splitting method to obtain the convergence to the equilibrium with quantitative estimates. Here the key part is making full use of the   cancellation due to the coupling effect of the system.

\subsection{Uniform-in-time bounds and the dissipation inequality}  By the energy identity and the coupling effect of the system, in this subsection, we will prove the uniform-in-time bounds for the regularity and then derive the dissipation inequality which is crucial to obtain the decay estimate.

\subsubsection{Zero-order estimate for the system}

 In what follows, we set $a\eqdefa\rho-1$ and
$\theta\eqdefa T-1.$
We first recall the basic energy identities for \eqref{cns}:
\begin{prop}\label{energylem1}
	Let $(\rho, u,T)$ be a global and smooth solution of \eqref{cns} and $T>0.$ Then the following equality holds
	\begin{equation}\label{energyest2}
	\begin{split}
	\frac{d}{dt}&\Big( \int (\rho\ln\rho-\rho+1)\,dx+\f12\int \rho |u|^2\,dx+\int \rho(T-\ln T-1)\,dx\Big)\\
	&+ \int\bigl(\f{\mu|\na u+(\na u)'|^2+\lam(\div u)^2}{T}+\f{|\na T|^2}{T^2}\bigr)\,dx=0.
	\end{split}
	\end{equation}
	\end{prop}
	
\begin{proof}
By the directly computation, we can get that
\beno
\f{d}{dt}\int (\rho\ln\rho-\rho+1)\,dx +\int \rho \div u\,dx=0,
\eeno
\beno
\f12\f{d}{dt}\int \rho |u|^2\,dx+\mu\int |\na u|^2\,dx+(\mu+\lam)\int (\div u)^2\,dx+\int \na(\rho T)\cdot u\,dx=0,
\eeno
\beno
\f{d}{dt}\int \rho T\,dx+\int \rho T\div u\,dx=\mu\int |\na u+(\na u)'|^2\,dx+\lam\int (\div u)^2\,dx,
\eeno
\beno
-\f{d}{dt}\int \rho(\ln T-1)\,dx= \int \rho \div u\,dx-\int \f{\mu|\na u+(\na u)'|^2+\lam(\div u)^2}{T}\,dx-\int \f{\Delta T}{T}\,dx.
\eeno
Combining all the above estimates, and noting that
$$\mu\int |\na u|^2\,dx+(\mu+\lam)\int (\div u)^2\,dx=\mu\int |\na u+(\na u)'|^2\,dx+\lam\int (\div u)^2\,dx,$$
$$\int \f{\Delta T}{T}\,dx=-\int \na T \cdot \na (\f1T)\,dx =\int \f{|\na T|^2}{T^2}\,dx,$$
we arrive at \eqref{energyest2}.
\end{proof}

\medskip

\begin{rmk}\label{rmk1}
	By Taylor expansion, it is not difficult to check that
$\rho\ln\rho-\rho+1\geq C(M_1)(\rho-1)^2$ and $T-\ln T-1\geq C(M_1)(T-1)^2$ if $(\rho, T)$ satisfies \eqref{Ass1}.
Moreover, if $T>0$, we have
\beno
\|\na u\|_{L^2}+\|\na T\|_{L^2} \leq  \|\sqrt{T}\|_{L^\infty} \|\f{\na u}{\sqrt{T}}\|_{L^2}+ \|T\|_{L^\infty} \|\f{\na T}{T}\|_{L^2} .
\eeno
Thus, from this lemma, we get that
\beno
\|a\|_{L^2}+\|\sqrt{\rho} u\|_{L^2}+\|\sqrt{\rho} \theta\|_{L^2}+\|\na u\|_{L^2_t(L^2)}+\|\na \theta\|_{L_t^2(L^2)} <\infty.
\eeno

\end{rmk}

\medskip

\begin{prop}\label{energylem2}
	Let $\mu>\f12 \lam,$ and $(\rho, u,T)$ be a global solution of \eqref{cns} and satisfy \eqref{Ass1}. Then the following inequality hold
	\begin{equation}\label{energyest3}
	\frac{d}{dt}\int \rho u^4\,dx+ \int |u|^2|\na u|^2\,dx
	\leq  C\|\na u\|_{L^2}^2( \| a\|_{L^6}^2 +\|\na \theta\|_{L^2}^2+\|\na u\|_{L^2}^2),
	\end{equation}
	where $C$ is a positive constants depending on $\mu,$ $\lam$ and $M_1$.
\end{prop}
\begin{proof}
	Multiplying $4|u|^2 u$ to $\eqref{fcns}_2$, and then integrating on $\R^3,$ we can obtain that
	\beno
	\begin{split}
		\frac{d}{dt}&\int \rho u^4\,dx + \int \Big[4|u|^2\Big( \mu|\na u|^2+(\lam+\mu)(\div u)^2+2\mu\big|\na|u|\big|^2  \Big)+4(\lam+\mu)(\na |u|^2)\cdot u\div u    \Big]\,dx\\
		&=4\int \div (|u|^2 u) (P-1)\,dx\leq C\int (P-1) |u|^2|\na u|\,dx
		\leq C\|P-1\|_{L^6}\|u^2\|_{L^3}\|\na u\|_{L^2}\\
		&\leq C\|  aT+\theta\|_{L^6}\|\na u\|_{L^2}^3\leq (\|a\|_{L^6}+\|\theta\|_{L^6})\|\na u\|_{L^2}^3\leq C(\|a\|_{L^6}^2+\|\na \theta\|_{L^2}^2)\|\na u\|_{L^2}^2+\|\na u\|_{L^2}^4.
	\end{split}
	\eeno
	Using the inequality $\big|\na|u|\big|\leq |\na u|$, we have
	\beno
	\begin{split}
		&4|u|^2\Big( \mu|\na u|^2+(\lam+\mu)(\div u)^2+2\mu\big|\na|u|\big|^2  \Big)+4(\lam+\mu)(\na |u|^2)\cdot u\div u\\
		\geq& 4|u|^2\Big[ \mu|\na u|^2+(\lam+\mu)(\div u)^2+2\mu\big|\na|u|\big|^2  -2(\lam+\mu)\big|\na|u|\big| |\div u|\Big]\\
		=& 4|u|^2\Big[ \mu|\na u|^2+(\lam+\mu)\Big( \div u- \big|\na|u|\big|\Big)^2\Big]+4|u|^2(\mu-\lam)\big|\na|u|\big|^2\\
		\geq&  C|u|^2|\na u|^2,
	\end{split}
	\eeno
	where in the last step we use $\mu>\f12\lam.$ Combining these two estimates, we arrive at  \eqref{energyest3}.
\end{proof}

\subsubsection{First-order estimate for the system}
We want to give the first order energy estimate for the system. First, we need following three lemmas which are obtained due to   the coupling effect of the system. They will
play the crucial role in the proof of main theorem.
\begin{lem}\label{energylem3}
Let $(\rho,u,T)$ be a smooth global solution of \eqref{fcns} and satisfy \eqref{Ass1}. Then the following estimates hold
\begin{equation}\label{divest1}
\Big|\int a\div u\,dx+\f{d}{dt}\int f(a)\,dx\Big|\leq C\|a\|_{L^6}\|\na u\|_{L^2},
\end{equation}
\begin{equation}\label{divest2}
\Big|\int \theta\div u\,dx+\f12\f{d}{dt}\int \rho\theta^2\,dx\Big|\leq C(\|\na u\|_{L^2}^2+\|\na \theta\|_{L^2}^2),
\end{equation}
where $f(a)=a-\ln(1+a),$ and $C$ is a constant just depending only on $\mu,$ $\lam$ and $M_1.$
\end{lem}
\begin{rmk}\label{rmk2}
It is not difficult to verify that there exists a small constant $c$ such that $|f(a)|\leq ca^2.$ And thanks to Remark \ref{rmk1}, $\int \rho\theta^2\,dx$ can be controlled by $\int \rho(T-\ln T-1)\,dx.$
\end{rmk}
\begin{proof}
Noting that $\eqref{fcns}_1,$ we can get that
\beno
\begin{split}
\int a\div u\,dx&=\int \f{a}{\rho}\rho\div u\,dx=-\int \f{a}{\rho}(a_t+u\cdot\na a)\,dx\\
&=-\f{d}{dt}\int f(a)\,dx-\int u\cdot\na f(a)\,dx=-\f{d}{dt}\int f(a)\,dx+\int f(a)\div u\,dx,
\end{split}
\eeno
which implies that
\beno
\begin{split}
&\Big|\int a\div u\,dx+\f{d}{dt}\int f(a)\,dx\Big|= \Big|\int f(a)\div u\,dx\Big|\\
&\leq C\|f(a)\|_{L^2}\|\na u\|_{L^2}\leq C\|a\|_{L^6}\|\na u\|_{L^2},
\end{split}
\eeno
where we use Remark \ref{rmk2} in the last step. It completes the proof of \eqref{divest1}.

The third equation of \eqref{fcns} can be rewritten by
\begin{equation}\label{eqoftheta}
\rho\theta_t+\rho u\cdot \na\theta+(P-1)\div u-\f{\mu}2|\na u+(\na u)'|^2-\lam(\div u)^2-\Delta \theta=-\div u.
\end{equation}
Then making the inner product to the above equation with $\theta,$ and noting that $(a,\theta)\in L^\infty((0,\infty);L^2\cap L^\infty),$ we obtain that
\beno
\begin{split}
\pm\int \theta \div u\,dx &\leq \mp\int \rho\theta\theta_t\,dx+C\|\rho\theta\|_{L^3}\| u\|_{L^6}\|\na \theta\|_{L^2}+C\|P-1\|_{L^3}\| \theta\|_{L^6}\|\na u\|_{L^2}+C\|\na u\|_{L^2}^2+\|\na \theta\|_{L^2}^2\\
&\leq \mp\int \rho\theta\theta_t\,dx+C(\|\na u\|_{L^2}^2+\|\na \theta\|_{L^2}^2).
\end{split}
\eeno
And
\beno
\begin{split}
\mp\int \rho\theta\theta_t\,dx&=\mp\f12\f{d}{dt}\int \rho \theta^2\,dx\pm\f12\int \rho_t \theta^2\,dx=\mp\f12\f{d}{dt}\int \rho \theta^2\,dx\mp\f12\int \theta^2u\cdot \na a\,dx\mp\f12\int \theta^2\rho\div u\,dx\\
&\leq \mp\f12\f{d}{dt}\int \rho \theta^2\,dx\pm\f12\int a\div(\theta^2 u)\,dx +C\|\rho\theta\|_{L^3}\|\theta\|_{L^6}\|\na u\|_{L^2}\\
&\leq \mp\f12\f{d}{dt}\int \rho \theta^2\,dx+C(\|\na u\|_{L^2}^2+\|\na \theta\|_{L^2}^2).
\end{split}
\eeno
Combining these two estimates, we arrive at \eqref{divest2}.
\end{proof}

\smallskip

\begin{lem}\label{energylem4}
Let $(\rho,u,T)$ be a smooth global solution of \eqref{fcns} and satisfy \eqref{Ass1}. Denote the effective viscosity flux as $G=(2\mu+\lam)\div u-(P-1).$ Then the following estimate holds
\begin{equation}\label{divest3}
-\int Ta_t\div u\,dx\leq -\f{1}{2\mu+\lam}\f{d}{dt}\int F(a)\,dx+C(\|\na u\|_{L^2}^2+\|\na \theta\|_{L^2}^2+\|u\na u\|_{L^2}^2+\|\na u\|_{L^2}\|\na G\|_{L^2}),
\end{equation}
where $F(a)=\f12 a^2+a-(1+a)\ln (1+a).$ And $C$ is a constant just depending only on $\mu,$ $\lam$ and $M_1$.
\end{lem}
\begin{rmk}\label{rmk3}
Thanks to the upper bound of density, we have that $|F(a)|\lesssim \rho\ln\rho-\rho+1.$
\end{rmk}
\begin{proof}
Using the first equation of \eqref{fcns}, we can get that
\begin{equation}\label{divest4}
\begin{split}
&-\int Ta_t\div u\,dx=\int T(\div u)^2\,dx+\int T\div(au)\div u\,dx\\
&\leq C\|\na u\|_{L^2}^2-\int au\cdot\na \theta\,\div u\,dx-\int Ta u\cdot\na \div u\,dx\\
&\leq C\|\na u\|_{L^2}^2+C\|\na \theta\|_{L^2}\|u\na u\|_{L^2}-\int Ta u\cdot\na \div u\,dx.
\end{split}
\end{equation}
By the definition of effective viscosity $G,$ we can separate the last term into four parts
\beno
\begin{split}
-\int Ta u\cdot\na \div u\,dx&=-\f1{2\mu+\lam}\int Ta u\cdot\na G\,dx-\f1{2\mu+\lam}\int Ta u\cdot\na (P-1)\,dx\\
&=-\f1{2\mu+\lam}\int Ta u\cdot\na G\,dx-\f1{2\mu+\lam}\int Ta u\cdot(\rho\na\theta+\theta\na a+\na a)\,dx\\
&=\f1{2\mu+\lam}\sum_{i=1}^4 I_i.
\end{split}
\eeno
\underline{Estimates of $I_1$ and $I_2$.} It is easy to obtain that
\begin{equation}\label{divest5}
\begin{split}
I_1+I_2\leq C\|a\|_{L^3}\|u\|_{L^6}(\|\na G\|_{L^2}+\|\na \theta\|_{L^2})\leq C\|\na u\|_{L^2}(\|\na G\|_{L^2}+\|\na \theta\|_{L^2}).
\end{split}
\end{equation}
\underline{Estimate of $I_3$.} Using integration by parts, we get that
\begin{equation}\label{divest6}
\begin{split}
I_3&=-\f12\int T\theta u\cdot \na (a^2)\,dx=\f12\int a^2\div(T\theta u)\,dx\\
&\leq C\|a\|_{L^3}(\|\theta\|_{L^6}\|\na u\|_{L^2}+\|u\|_{L^6}\|\na \theta\|_{L^2})\leq C\|\na u\|_{L^2}\|\na \theta\|_{L^2}.
\end{split}
\end{equation}
\underline{Estimate of $I_4$.} Recalling $T=\theta+1,$ we have
\beno
I_4=- \bigl(\int \theta au\cdot \na a\,dx+\int au\cdot \na a\,dx\bigr)=\sum_{i=1}^2 I_{4,i}.
\eeno
We can estimate $I_{4,1}$ similarly as $I_3$ as follow:
\begin{equation}\label{divest7}
\begin{split}
I_{4,1}&=-\f12\int \theta u\cdot \na (a^2)\,dx=\f12\int a^2\div(\theta u)\,dx\\
&\leq C\|a\|_{L^3}(\|\theta\|_{L^6}\|\na u\|_{L^2}+\|u\|_{L^6}\|\na \theta\|_{L^2})\leq C\|\na u\|_{L^2}\|\na \theta\|_{L^2}.
\end{split}
\end{equation}
For $I_{4,2},$ we can obtain that
\begin{equation}\label{divest8}
\begin{split}
I_{4,2}&=-\int \rho u\cdot (\f{a}{1+a} \na a)\,dx=-\int \rho u\cdot\na f(a)\,dx\\
&=\int f(a)\div(\rho u)\,dx=-\int f(a) a_t\,dx=-\f{d}{dt}\int F(a)\,dx,
\end{split}
\end{equation}
where $f(a)$ is defined in Lemma \ref{energylem3}.

Putting \eqref{divest5}-\eqref{divest8} into \eqref{divest4}, we arrive at \eqref{divest3}.
\end{proof}

\smallskip

\begin{lem}\label{energylem5}
Let $(\rho,u,T)$ be a smooth global solution of \eqref{fcns} and satisfy \eqref{Ass1}. Then the following estimate holds
\begin{equation}\label{divest9}
\begin{split}
-\int \rho\theta_t\div u\,dx&\leq \f{1}{2\mu+\lam}\f{d}{dt}\int H(a,\theta)\,dx+C\|\sqrt{\rho} u_t\|_{L^2}(\|u\na u\|_{L^2}+\|\na u\|_{L^2})\\
&\quad +C\|\na G\|_{L^2}\bigl(\|\na u\|_{L^2}+\|u\na u\|_{L^2}+\|\na \theta\|_{L^2}\bigl)+C\|a\|_{L^6}\|\na u\|_{L^2}\\
&\quad +C\bigl(\|\na u\|_{L^2}^2+\|u\na u\|_{L^2}^2+\|\na \theta\|_{L^2}^2\bigl),
\end{split}
\end{equation}
where $H(a,\theta)=\rho\theta(\f12\theta-\f12a\theta-a)-f(a),$ and $f(a)$ is defined in Lemma \ref{energylem3}. $C$ is a constant just depending only on $\mu,$ $\lam$ and $M_1.$
\end{lem}
\begin{rmk}\label{rmk4}
Thanks to the upper bound of density and temperature, we have that $|H(a,\theta)|\lesssim (\rho\ln\rho-\rho+1)+\rho(T-\ln T-1).$
\end{rmk}
\begin{proof}
Recalling the definition of effective viscosity $G,$ we can obtain that
\begin{equation}\label{divest10}
\begin{split}
-\int \rho\theta_t\div u\,dx&=-\f1{2\mu+\lam}\int \rho\theta_t(a\theta+a+\theta+G)\,dx\\
&=-\f1{2\mu+\lam}\Bigl[ \int a\rho \theta\theta_t\,dx+\int a^2\theta_t \,dx+\int a\theta_t\,dx +\int \rho\theta_t(\theta+G)\,dx \Bigr]\\
&=\f1{2\mu+\lam}\sum_{i=1}^4 II_i.
\end{split}
\end{equation}
\underline{Estimate of $II_1$.} Using $\eqref{fcns}_1,$ we have that
\begin{equation}\label{divest11}
\begin{split}
II_1&=-\f12\int a\rho\pa_t(\theta^2)\,dx=-\f12\f{d}{dt}\int a\rho\theta^2\,dx+\f12\int \theta^2(\rho+a)a_t\,dx\\
&=-\f12\f{d}{dt}\int a\rho\theta^2\,dx-\f12\int \theta^2(\rho+a)(u\cdot\na a+\rho\div u)\,dx\\
&\leq -\f12\f{d}{dt}\int a\rho\theta^2\,dx-\f12\int \theta^2u\cdot\na(a+a^2)\,dx+C\|\rho\theta\|_{L^3}\|\theta\|_{L^6}\|\na u\|_{L^2}\\
&\leq -\f12\f{d}{dt}\int a\rho\theta^2\,dx+\f12\int(a+a^2)\div(\theta^2 u)\,dx+C(\|\na\theta\|_{L^2}^2+\|\na u\|_{L^2}^2)\\
&\leq -\f12\f{d}{dt}\int a\rho\theta^2\,dx+C(\|\na\theta\|_{L^2}^2+\|\na u\|_{L^2}^2).
\end{split}
\end{equation}
\underline{Estimate of $II_2$.} Using $\eqref{fcns}_1$ again, we have
\begin{equation}\label{divest12}
\begin{split}
II_2&=-\f{d}{dt}\int a^2\theta\,dx+2\int \theta aa_t\,dx=-\f{d}{dt}\int a^2\theta\,dx-2\int a\theta(u\cdot\na a+\rho\div u)\,dx\\
&\leq -\f{d}{dt}\int a^2\theta\,dx-\int \theta u\cdot\na(a^2)\,dx+C\|a\|_{L^3}\|\theta\|_{L^6}\|\na u\|_{L^2}\\
&\leq -\f{d}{dt}\int a^2\theta\,dx+\int a^2\div(\theta u)\,dx+C(\|\na\theta\|_{L^2}^2+\|\na u\|_{L^2}^2)\\
&\leq -\f{d}{dt}\int a^2\theta\,dx+C(\|\na\theta\|_{L^2}^2+\|\na u\|_{L^2}^2).
\end{split}
\end{equation}
\underline{Estimate of $II_3$.} Using $\eqref{fcns}_1$ and \eqref{divest2}, we can obtain that
\begin{equation}\label{divest13}
\begin{split}
II_3&=-\f{d}{dt}\int a \theta\,dx+\int \theta a_t\,dx=-\f{d}{dt}\int a\theta\,dx-\int \theta\bigl(\div(au)+\div u\bigr)\,dx\\
&=-\f{d}{dt}\int a\theta\,dx+\int au\cdot\na\theta\,dx-\int \theta\div u\,dx\\
&\leq \f{d}{dt}\int \bigl(\f12 \rho\theta^2-a\theta\bigr)\,dx+C\|a\|_{L^3}\|u\|_{L^6}\|\na \theta\|_{L^2}+C(\|\na\theta\|_{L^2}^2+\|\na u\|_{L^2}^2)\\
&\leq \f{d}{dt}\int \bigl(\f12 \rho\theta^2-a\theta\bigr)\,dx+C(\|\na\theta\|_{L^2}^2+\|\na u\|_{L^2}^2).
\end{split}
\end{equation}
\underline{Estimate of $II_4$.}  Recalling $\rho E=\rho T+\f12\rho|u|^2,$ we have
\begin{equation}\label{defII4}
II_4=-\int (\rho E)_t(\theta+G)\,dx+\int (\f12\rho|u|^2)_t(\theta+G)\,dx+\int T\rho_t(\theta+G)\,dx=\sum_{i=1}^3 II_{4,i}.
\end{equation}
For $II_{4,1},$ using $\eqref{cns}_3,$ integration by parts, and \eqref{divest1}, we can get that
\begin{equation}\label{divest14}
\begin{split}
II_{4,1}&=-2\int (P-1)u\cdot\na (\theta+G)\,dx +2\int (\theta+G)\div u\,dx-\int \f12\rho|u|^2u\cdot\na(\theta+G)\,dx\\
&\quad+\int \bS(u) u\na(\theta+G)\,dx+\int \na\theta\cdot\na(\theta+G)\,dx\\
&\leq 2\int \div u\bigl((2\mu+\lam)\div u+a+a\theta\bigr)\,dx-\int \f12\rho|u|^2u\cdot\na(\theta+G)\,dx\\
&\quad +\|\na(\theta+G)\|_{L^2}\bigl(\|P-1\|_{L^3}\|u\|_{L^6}+\|u\na u\|_{L^2}+\|\na \theta\|_{L^2}\bigl)\\
&\leq -2\f{d}{dt}\int f(a)\,dx-\int \f12\rho|u|^2u\cdot\na(\theta+G)\,dx+ C\|\na G\|_{L^2}\bigl(\|\na u\|_{L^2}+\|u\na u\|_{L^2}+\|\na \theta\|_{L^2}\bigl)\\
&\quad +C\bigl(\|\na u\|_{L^2}^2+\|u\na u\|_{L^2}^2+\|\na \theta\|_{L^2}^2+\|a\|_{L^6}\|\na u\|_{L^2}\bigl).
\end{split}
\end{equation}
For $II_{4,2},$ using $\eqref{cns}_1,$ we can obtain that
\begin{equation}\label{divest15}
\begin{split}
II_{4,2}&=\int \f12 \rho_t |u|^2(\theta+G)\,dx +\int \rho u\cdot u_t (\theta+G)\,dx= -\int \f12\div(\rho u)|u|^2(\theta+G)\,dx+\int \rho u\cdot u_t (\theta+G)\,dx\\
&=\int \f12 \rho|u|^2 u\cdot\na (\theta+G)\,dx+\int \rho(u\cdot\na u+u_t)\cdot u(\theta+G)\,dx\\
&= \int \f12 \rho|u|^2 u\cdot\na (\theta+G)\,dx+\int \rho(u\cdot\na u+u_t)\cdot u\bigl( (2\mu+\lam)\div u+aT\bigr)\,dx\\
&\leq \int\f12 \rho|u|^2 u\cdot\na (\theta+G)\,dx +C(\|u\na u\|_{L^2}+\|\sqrt{\rho} u_t\|_{L^2})(\|u\na u\|_{L^2}+\|u\|_{L^6}\|aT\|_{L^3})\\
&\leq \int\f12 \rho|u|^2 u\cdot\na (\theta+G)\,dx +C(\|u\na u\|_{L^2}+\|\sqrt{\rho} u_t\|_{L^2})(\|u\na u\|_{L^2}+\|\na u\|_{L^2}).
\end{split}
\end{equation}
For $II_{4,3},$ using $\eqref{cns}_1$ again, and \eqref{divest1}, we can have that
\begin{equation}\label{divest16}
\begin{split}
II_{4,3}&=-\int T(\theta+G)\div(au) \,dx -\int (\theta+G)\div u \,dx-\int \theta (\theta+G)\div u \,dx\\
&\leq \int au\cdot \na [T(\theta +G)]\,dx-\int (\theta+G)\div u\,dx+C(\|a\|_{L^3}\|\theta\|_{L^6}+\|\na u\|_{L^2})\|\na u\|_{L^2}\\
&\leq \int Tau\cdot\na(\theta+G)\,dx +\int (au\cdot\na \theta-\div u)\bigl((2\mu+\lam)\div u-a-a\theta\bigr)\,dx+C\|\na \theta \|^2_{L^2}+C\|\na u\|^2_{L^2}\\
&\leq \f{d}{dt}\int f(a)\,dx+C(\|\na G\|_{L^2}+\|a\|_{L^6})\|\na u\|_{L^2}+C(\|\na u\|_{L^2}^2+\|u\na u\|_{L^2}^2+\|\na \theta\|_{L^2}^2).
\end{split}
\end{equation}
Plugging \eqref{divest14}, \eqref{divest15}, and \eqref{divest16} into \eqref{defII4}, we can obtain that
\begin{equation}\label{divest17}
\begin{split}
II_4&\leq -\f{d}{dt}\int f(a)\,dx+ C\|\na G\|_{L^2}\bigl(\|\na u\|_{L^2}+\|u\na u\|_{L^2}+\|\na \theta\|_{L^2}\bigl)+C\|\sqrt{\rho} u_t\|_{L^2}(\|u\na u\|_{L^2}+\|\na u\|_{L^2})\\
&\quad +C\bigl(\|\na u\|_{L^2}^2+\|u\na u\|_{L^2}^2+\|\na \theta\|_{L^2}^2+\|a\|_{L^6}\|\na u\|_{L^2}\bigl).
\end{split}
\end{equation}
Putting \eqref{divest11}, \eqref{divest12}, \eqref{divest13}, and \eqref{divest17} into \eqref{divest10}, we arrive at \eqref{divest9}.
\end{proof}

\medskip

Now we are in the position to give the first order energy estimate. More precisely, we have the following proposition.
\begin{prop}\label{energyprop}
Let $\mu>\frac12 \lam$ and $(\rho,u,T)$ be a smooth global solution of \eqref{fcns} and satisfy \eqref{Ass1}. Then $a \in L^\infty((0,+\infty);L^2\cap L^6)\cap L^2((0,+\infty); L^6),$ and $u\cdot\na u\in L^2((0,+\infty); L^2),$ $u\in L^\infty((0,+\infty);L^4\cap H^1 )\cap L^2((0,+\infty); \dot{H}^1\cap \dot{H}^2),$ $\theta \in L^\infty((0,+\infty);L^2 )\cap L^2((0,+\infty); \dot{H}^1).$ Furthermore, the following inequality holds
\begin{equation}\label{energyest1}
\begin{split}
\f{d}{dt}&\Big[ A_1\|\rho^{\f14} u\|_{L^4}^4+A_2\Big( \mu\|\na u\|_{L^2}^2+(\lam+\mu)\|\div u\|_{L^2}^2-\int (P-1)\div u\,dx- \f1{2\mu+\lam}\int  \bigl(H(a,\theta)-F(a)\bigr)\,dx\Big)\\
&+A_3\|a \|_{L^6}^2+A_4\big(\int (\rho\ln\rho-\rho+1)\,dx+\|\sqrt{\rho} u\|_{L^2}^2+\int\rho(T-\ln T-1)\,dx\big) \Big]\\
&\quad+ A_5\Big(\|\sqrt{\rho}u_t\|_{L^2}^2+\| |u| |\na u|\|_{L^2}^2+\|\na u\|_{L^2}^2 + \|\na\theta\|_{L^2}^2+\| a \|_{L^6}^2\\
&\qquad +\|\na G\|_{L^2}^2+\|\na\curl u\|_{L^2}^2+\|\na u\|_{L^6}^2+\|u\|_{L^\infty}^2\Big)  \leq0,
\end{split}
\end{equation}
where $A_i(i=1,\cdots,5)$ are positive constants depending on $\mu,$ $\lam$ and $M_1.$
\end{prop}
\begin{rmk} Thanks to the energy identity, choose $A_4$ large enough and then we can derive that
\beno
&&A_1\|\rho^{\f14} u\|_{L^4}^4+A_2\Big( \mu\|\na u\|_{L^2}^2+(\lam+\mu)\|\div u\|_{L^2}^2-\int (P-1)\div u\,dx- \f1{2\mu+\lam}\int  \bigl(H(a,\theta)-F(a)\bigr)\,dx\Big)\\
&&+A_3\|a \|_{L^6}^2+A_4\big(\int (\rho\ln\rho-\rho+1)\,dx+\|\sqrt{\rho} u\|_{L^2}^2+\int\rho(T-\ln T-1)\,dx\big)\\
&&\sim \|\rho^{\f14} u\|_{L^4}^4+\|\na u\|_{L^2}^2+\int (\rho\ln\rho-\rho+1)\,dx+\|\sqrt{\rho} u\|_{L^2}^2+\int\rho(T-\ln T-1)\,dx+\|  a \|_{L^6}^2.
\eeno
\end{rmk}

\begin{proof} To derive the desired results, we split the proof into several steps.
	
	{\it Step 1: Estimate of $\na u$.} First, taking the inner product of $\eqref{fcns}_2$ and $u_t,$ we get that
\begin{equation}\label{enep1}
\begin{split}
\f{d}{dt}&(\f12 \mu\|\na u\|_{L^2}^2+\f12(\lam+\mu)\|\div u\|_{L^2}^2)+\int \rho |u_t|^2\,dx\\
&=\f{d}{dt}\int(P-1)\div u\,dx-\int (P-1)_t\div u\,dx-\int \rho u\cdot \na u\cdot u_t\,dx\\
&\leq \f{d}{dt}\int(P-1)\div u\,dx-\int Ta_t\div u\,dx-\int \rho\theta_t\div u\,dx +C\|\sqrt{\rho}u_t\|_{L^2}\|u\na u\|_{L^2}.
\end{split}
\end{equation}

Plugging \eqref{divest3} and \eqref{divest9} into \eqref{enep1}, together with Young inequality, we obtain that
\begin{equation}\label{enep2}
\begin{split}
\f{d}{dt}&\bigg(\f12 \mu\|\na u\|_{L^2}^2+\f12(\lam+\mu)\|\div u\|_{L^2}^2-\int(P-1)\div u\,dx- \f1{2\mu+\lam}\int  \bigl(H(a,\theta)-F(a)\bigr)\,dx\bigg)+ C\|\rho^{\f12} u_t\|_{L^2}^2\\
&\leq  C(\|\na u\|_{L^2}^2+\|u\cdot\na u\|_{L^2}^2+\|\na \theta\|_{L^2}^2+\|a\|_{L^6}\|\na u\|_{L^2})+C\eta \|\na G\|_{L^2}^2,
\end{split}
\end{equation}
where $\eta$ is a small constant, and the constant $C$ depends on $\mu,$ $\lam$ and $M_1.$

{\it Step 2: Improving estimate by the elliptic system.} Taking div and curl on both side of $\eqref{fcns}_2,$ we can get that
\begin{equation}
  \left\{
    \begin{aligned}
      & \Delta G=\div( \rho u_t+\rho u\cdot\na u),\\
      &-\mu\Delta (\curl u)=\curl( \rho u_t+\rho u\cdot\na u).
        \end{aligned}
  \right.
  \label{eq:BD}
\end{equation}
By the standard elliptic estimate, we have
\begin{equation}\label{enep3}
\|\na G\|_{L^2}^2+\|\na\curl u \|^2_{L^2}\leq C( \|\sqrt{\rho} u_t\|_{L^2}^2+ \|u\cdot\na u\|_{L^2}^2).
\end{equation}
Combining \eqref{enep2} and \eqref{enep3}, and choosing $\eta$ small enough, we can get that
\begin{equation}\label{enep4}
\begin{split}
\f{d}{dt}& \Big( \mu\|\na u\|_{L^2}^2+(\lam+\mu)\|\div u\|_{L^2}^2-\int(P-1)\div u\,dx- \f1{2\mu+\lam}\int  \bigl(H(a,\theta)-F(a)\bigr)\,dx\Big)\\
&+ \big(\|\sqrt{\rho}  u_t\|_{L^2}^2 +\|\na G\|_{L^2}^2+\|\na\curl u \|^2_{L^2}  \big) \leq C( \|\na u\|_{L^2}^2 +\|u\cdot\na u\|_{L^2}^2+\|\na \theta\|_{L^2}^2+\|a\|_{L^6}\|\na u\|_{L^2}),
\end{split}
\end{equation}
where $C$ is a positive constant depending on $\mu,$ $\lam$ and $M_1.$

{\it Step 3: Estimate of $a$.}  The first equation of \eqref{cns} can be rewritten by
\begin{equation}\label{enep5}
a _t+u\cdot\na a +\f1{2\mu+\lam}a +a \div u=-(\div u-\f1{2\mu+\lam}a ).
\end{equation}
Then making the inner product  to the above equation with $  |a |^4 a $, we obtain that
\beno
\f16\f{d}{dt}\|a \|_{L^6}^6+\f{1}{2\mu+\lam}\|a \|_{L^6}^6 +\f56\int \div u |a |^6\,dx\leq C\|\div u-\f1{2\mu+\lam}a \|_{L^6}\|a ^5\|_{L^{\f65}},
\eeno
which implies
\beno
\f16\f{d}{dt}\|a \|_{L^6}^6 + \f1{2\mu+\lam} \int (1+\f56 a)a^6\,dx& \leq & C \|\div u-\f1{2\mu+\lam}a \|_{L^6}\|a ^5\|_{L^{\f65}}\\
& \leq & C \|\div u-\f1{2\mu+\lam}a \|_{L^6}\|a\|^5_{L^6}.
\eeno
Dividing the above estimate by $\|a \|_{L^6}^4,$ and recalling $1+\f56 a\geq \f16$, we get that
\ben\label{est:a}
\f{d}{dt}\|a \|_{L^6}^2 +  \| a \|_{L^6}^2     \leq  C \|\div u-\f1{2\mu+\lam}a \|_{L^6}^2 .
\een
Noting that $\div u-\f1{2\mu+\lam}a=\f{G+\rho\theta}{2\mu+\lam},$ we arrive that
\ben\label{est:a1}
\f{d}{dt}\|a \|_{L^6}^2 +  \| a \|_{L^6}^2     \leq  C (\|G\|_{L^6}^2+\|\rho\theta\|_{L^6}^2)\leq C(\|\na G\|_{L^2}^2+\|\na\theta\|_{L^2}^2) .
\een

{\it Step 4: Closing the energy estimates.}  Combining \eqref{energyest2}, \eqref{energyest3}, \eqref{enep4} and \eqref{est:a1}, we can get
\begin{equation}\label{enep8}
\begin{split}
\f{d}{dt}&\Big[ A_1\|\rho^{\f14} u\|_{L^4}^4+A_2\Big( \mu\|\na u\|_{L^2}^2+(\lam+\mu)\|\div u\|_{L^2}^2-\int (P-1)\div u\,dx- \f1{2\mu+\lam}\int  \bigl(H(a,\theta)-F(a)\bigr)\,dx\Big)\\
&+A_3\|a \|_{L^6}^2+A_4\big(\int (\rho\ln\rho-\rho+1)\,dx+\|\sqrt{\rho} u\|_{L^2}^2+\int\rho(T-\ln T-1)\,dx\big) \Big]\\
&\quad+ A_5\Big(\|\sqrt{\rho}u_t\|_{L^2}^2+\| |u| |\na u|\|_{L^2}^2+\|\na u\|_{L^2}^2 + \|\na\theta\|_{L^2}^2+\|\na G\|_{L^2}^2+\|\na\curl u\|_{L^2}^2+\| a \|_{L^6}^2\Big) \\
&\qquad\leq A_6( \| a \|_{L^6}^2 +\|\na u\|_{L^2}^2+\|\na\theta\|_{L^2}^2) \|\na u\|_{L^2}^2,
\end{split}
\end{equation}
where $A_i(i=1,\cdots,6)$ are positive constants depending on $\lam,$ $\mu$ and $M_1$, and which ensure that the term $A_2\Bigl(\int (P-1)\div u\,dx+ \f1{2\mu+\lam}\int  \bigl(H(a,\theta)-F(a)\bigr)\,dx\Bigr)$ can be controlled by $A_2(\lam+\mu)\|  \div u\|_{L^2}^2$ and $A_4\Bigl(\int (\rho\ln\rho-\rho+1)\,dx+\int\rho(T-\ln T-1)\,dx\Bigr).$ By Gronwall's inequality, the above estimate ensures that $u\in L^\infty((0,+\infty);L^4\cap H^1)\cap L^2((0,+\infty); \dot{H}^1)$, $u_t\in L^2((0,+\infty); L^2),$ $a \in L^\infty((0,+\infty); L^2)\cap L^2((0,+\infty); L^6),$ $u\cdot\na u\in L^2((0,+\infty); L^2),$ $\theta \in L^\infty((0,+\infty);L^2)\cap L^2((0,+\infty); \dot{H}^1).$

Using these estimates, we can improve the estimate \eqref{enep8}.  Notice that the term in the righthand side of  \eqref{enep8} can be bounded by $C(\|\na u\|_{L^2}^2+\|\na \theta\|_{L^2}^2)$. Then thanks to the energy identity \eqref{energyest2},  the dissipation inequality in the proposition is followed by the fact that for $i\ge1$ and $p\in[2,6],$
\begin{equation}\label{controlofu1}
\begin{split}
 \|\na^i u\|_{L^p}&\leq \|\na^{i-1} \curl u\|_{L^p}+ \|\na^{i-1}\bigl(\div u-\f1{2\mu+\lam}(P-1) \bigr)\|_{L^p}+\f1{2\mu+\lam}\|\na^{i-1}(P-1) \|_{L^p}\\
\|\na^i u\|_{L^6}&\leq\|\na^i \curl u\|_{L^2}+ \f1{2\mu+\lam}\|\na^i G\|_{L^2}+\f1{2\mu+\lam}\|\na^{i-1}(P-1) \|_{L^6}.
\end{split}
\end{equation}
And the last estimate means $\na u\in L^2((0,+\infty); L^6).$ By interpolation, we have
$$\|u\|_{L^2(L^\infty)}^2\leq C\|\na u\|_{L^2(L^2)}\|\na u\|_{L^2(L^6)},$$
which completes the proof of \eqref{energyest1}.
\end{proof}

\medskip

\subsubsection{Improving regularity estimate for $u$ and $\theta$} In order to get the dissipation estimate for $a$, we first improve the regularity estimates for $u$ and $\theta$ in this subsection.  We still assume that $(\rho,u,T)$ is a global and smooth solution of \eqref{fcns}. We set up some notations. For a function or vector field (or even a $3\times 3$ matrix) $f(t,x)$, the material derivative $\dot{f}$ is defined by
\beno
\dot{f}= f_t+u\cdot\nabla f,
\eeno
and $\div(f\otimes u)= \sum_{j=1}^3\partial_j(fu_j)$.
For two matrices $A=(a_{ij})_{3\times 3}$ and $B=(b_{ij})_{3\times 3}$, we use the notation $A:B=\sum_{i,j=1}^3a_{ij}b_{ij}$ and $AB$ is as usual the multiplication of matrix.

\begin{prop}\label{Prop41}
Let $\mu>\frac12 \lam$ and $(\rho,u,T)$ be a global and smooth solution of \eqref{fcns} and satisfy \eqref{Ass1} and the admissible condition \eqref{admissible}. Then there exist constants $A_i$(i=1,\dots,6) such that
\begin{equation}\label{energyest4}
\begin{split}
&\frac{d}{dt}\Big[ A_1\|\rho^{\f14} u\|_{L^4}^4+A_2\Big( \mu\|\na u\|_{L^2}^2+(\lam+\mu)\|\div u\|_{L^2}^2-\int (P-1)\div u\,dx- \f1{2\mu+\lam}\int  \bigl(H(a,\theta)-F(a)\bigr)\,dx\Big)\\
&\quad+A_3\|  a \|_{L^6}^2+A_4\big(\int (\rho\ln\rho-\rho+1)\,dx+\|\sqrt{\rho} u\|_{L^2}^2+\int\rho(T-\ln T-1)\,dx\big)+A_5\Big(\|\sqrt{\rho}\dot{u}\|_{L^2}^2+\|\na\theta\|_{L^2}^2\\
&\qquad-\int \theta(2\mu|D u|^2+\lam(\div u)^2)\,dx\Big)\Big]
+A_6\Big( \|\sqrt{\rho}u_t\|_{L^2}^2+\| |u| |\na u|\|_{L^2}^2+\|\na u\|_{L^2}^2 + \|\na\theta\|_{L^2}^2 +\|\na G\|_{L^2}^2\\
&\quad\qquad+\|\na\curl u\|_{L^2}^2+\| a \|_{L^6}^2+\|\na u\|_{L^6}^2
+\|\na \dot{u}\|_{L^2}^2+\|\sqrt{\rho}\dot\theta\|_{L^2}^2+\|G\|_{W^{1,6}}^2+\|\curl u\|_{W^{1,6} }^2 +\|\Delta\theta\|_{L^2}^2 \Big)\leq 0.
\end{split}
\end{equation}

\end{prop}

\begin{proof} To derive the desired results, we split the proof into several steps.

{\it Step 1: Estimate of $\dot{u}$.} We rewrite the second equation of \eqref{fcns} as
\beno
\rho\dot{u}+\nabla P -\div\bS(u)=0.
\eeno
Then it is not difficult to check that
\ben\label{eq41}
\begin{split}
&\rho \dot{u}_t+\rho u\cdot \nabla \dot{u}+\nabla P_t+\div(\nabla P \otimes u)\\
&\quad=\mu\big[\Delta u_t+\div(\Delta u\otimes u)\big]+
(\lambda+\mu)\big[\nabla\div u_{t}+\div((\nabla\div u)\otimes u)\big].
\end{split}
\een
By the  energy estimate, we derive that
\ben\label{eq42}
&&\f{d}{dt}\int \f{1}{2}\rho|\dot{u}|^2\,dx\underbrace{-\mu\int \dot{u}\cdot\big(\Delta
u_t+\div(\Delta u\otimes u)\big)\,dx}_{\eqdefa III_1}\nonumber\\
&&\quad-(\lambda+\mu)\underbrace{\int \dot{u}\cdot\big((\nabla
\div u_{t})+\div((\nabla\div u)\otimes u))\big)\,dx}_{\eqdefa III_2}\\
&&=\underbrace{\int  P_{t}\div\dot{u}+(\dot{u}\cdot\nabla u)\cdot\nabla P\, dx}_{\eqdefa III_3}.\nonumber
\een
\underline{Estimate of $III_1$.} It is easy to check that
\beno
\begin{split}
&-\int \dot{u}\cdot\big(\Delta
u_t+\div(\Delta u\otimes u)\big)dx=\int \left[\nabla\dot{u}:\nabla u_t+ u\otimes\Delta u:\nabla \dot{u}\right]dx\\
&=\int \Big[|\nabla\dot{u}|^2-\big((\nabla u\nabla u)+(u\cdot\nabla) \nabla u\big):\nabla\dot{u}-\nabla(u\cdot\nabla\dot{u}):\nabla u\Big]dx\\
&=\int \Big[|\nabla\dot{u}|^2-(\nabla u\nabla u):\nabla\dot{u}+\big((u\cdot\nabla)\nabla\dot{u}\big):\nabla u
-(\nabla u\nabla\dot{u}):\nabla u-\big((u\cdot\nabla)\nabla\dot{u}\big):\nabla u\Big]dx\\
&\geq \int \left[\f{3}{4}|\nabla\dot{u}|^2-C|\nabla u|^4\right]dx.
\end{split}
\eeno
\underline{Estimate of $III_2$.} Observe that
\beno
\begin{split}
&\div\big((\nabla\div u)\otimes u\big)=\nabla(u\cdot\nabla\div u)-\div(\div u\nabla\otimes u)+\nabla(\div u)^2,\\
&\div\dot{u}=\div u_t+\div(u\cdot\nabla u)=\div u_t+u\cdot\nabla\div u+\nabla u:(\nabla u)',
\end{split}
\eeno
then we get
\beno
\begin{split}
&-\int \dot{u}\cdot\Big[\nabla\div
u_{t}+\div\big((\nabla\div u)\otimes u\big)\Big]dx\\
&=\int \Big[\div\dot{u}\div u_t+\div\dot{u}(u\cdot\nabla\div u)
-\div u(\nabla\dot{u})':\nabla u+\div\dot{u}(\div u)^2\Big]dx\\
&=\int \Big[|\div\dot{u}|^2-\div\dot{u}\nabla u:(\nabla u)'-\div u(\nabla\dot{u})^T:\nabla u+\div\dot{u}(\div u)^2\Big]dx\\
&\geq\int \Big[\f{1}{2}|\div\dot{u}|^2-\f{1}{4} |\nabla\dot{u}|^2-C|\nabla u|^4\Big]dx.
\end{split}
\eeno

\underline{Estimate of $III_3$.} We have
\beno
\begin{split}
III_3&=\int \bigl((\rho T)_t\div\dot{u}-P(\na u)':\na\dot{u}-\rho Tu\cdot\na\div\dot{u}\bigr)\, dx\\
&=\int \bigl((\rho T)_t\div\dot{u}+\div(\rho Tu)\div\dot{u}-P(\na u)':\na\dot{u}\bigr)\, dx=\int \bigl(\rho \dot{\theta}\div\dot{u}-P(\na u)':\na\dot{u}\bigr)\, dx\\
&\leq C\|\sqrt{\rho}\dot{\theta}\|_{L^2}\|\div \dot{u}\|_{L^2}+\|P\|_{L^\infty}\|\nabla u\|_{L^2  }\|\nabla\dot{u}\|_{L^2 } .
\end{split}
\eeno

Substituting these estimates into (\ref{eq42}) yields
\ben\label{eq43}
\begin{split}
&\f{d}{dt}\int \rho|\dot{u}|^2\,dx+\mu\int |\nabla\dot{u}|^2\,dx
+(\lambda+\mu)\int |\div\dot{u}|^2\,dx \\
&\leq C(\|\na u\|_{L^4}^4+\|\nabla {u}\|^2_{L^2 }+\|\sqrt{\rho}\dot{\theta}\|_{L^2}^2)\leq C(\|\na u\|_{L^4}^4+\|\nabla {u}\|^2_{L^2 }+\|\sqrt{\rho}\theta_t\|_{L^2}^2+\|u\cdot\na\theta\|_{L^2}^2).
\end{split}
\een

{\it Step 2: Estimate of $\theta_t$.} Multiply $\eqref{fcns}_3$ by $\theta_t$ and take integration to obtain
\beno
\begin{split}
&\f12\f{d}{dt}\int |\na\theta|^2\,dx+\int \rho\theta_t^2\,dx\\
=&-\int\rho u\cdot\na\theta \theta_t\,dx-\int\rho \theta\div u\theta_t\,dx+\int (2\mu|D u|^2+\lam(\div u)^2)\theta_t\,dx\\
\leq &C\|{\rho}^\f12\theta_t\|_{L^2}(\|u\cdot\na\theta\|_{L^2}+\|\na u\|_{L^2}) +\int (2\mu|D u|^2+\lam(\div u)^2)\theta_t\,dx\\
\leq &\eta\|{\rho}^\f12\theta_t\|_{L^2}^2+C_\eta(\|u\cdot\na\theta\|_{L^2}^2+\|\na u\|_{L^2}^2)+\int (2\mu|D u|^2+\lam(\div u)^2)\theta_t\,dx.
\end{split}
\eeno
Then we can estimate the last term in the following\beno
\begin{split}
&\int (2\mu|D u|^2+\lam(\div u)^2)\theta_t\,dx\\
=&\f{d}{dt}\int \theta(2\mu|D u|^2+\lam(\div u)^2)\,dx-\int \theta\big(4\mu (D u):(D u_t)+2\lam (\div u)(\div u_t)\big)\,dx\\
=&\f{d}{dt}\int \theta(2\mu|D u|^2+\lam(\div u)^2)\,dx-\int \theta\big(4\mu (D u):(D \dot u)+2\lam (\div u) (\div \dot u)\big)\,dx\\
&\quad+\int \theta\big(4\mu (D u):(D (u\cdot \na u)) +2\lam (\div u )(\div (u\cdot \na u))\big)\,dx\overset{def}{=}\sum_{i=1}^3IV_i.
\end{split}
\eeno
It is not difficult to derive that\beno
\begin{split}
IV_2
\leq& C\|\na u\|_{L^2}\|\na \dot u\|_{L^2}
\leq \eta\|\na \dot u\|_{L^2}^2+C_\eta\|\na u\|_{L^2}^2,
\end{split}
\eeno
\beno
\begin{split}
IV_3
=&\int \theta\big(4\mu (D u):  D(u\cdot \na) u +2\lam (\div u) \div (u\cdot \na) u)\big)\,dx\\
&\quad+\int \theta\big(4\mu (D u): (u\cdot \na (D u)) +2\lam (\div u)  (u\cdot \na (\div u))\big)\,dx\\
=&\int \theta\big(4\mu (D u): D(u\cdot \na) u +2\lam (\div u)  \div (u\cdot \na) u\big)\,dx-\int \big(\theta (\div u)+u\cdot \na \theta)(2\mu |D u|^2+\lam (\div u)^2\big)\,dx\\
\leq &C(\|\na u\|_{L^2}+\|\na \theta\|_{L^2})\|\na u\|_{L^4}^2
\leq C(\|\na u\|_{L^2}^2+\|\na \theta\|_{L^2}^2+\|\na u\|_{L^4}^4).
\end{split}
\eeno

Substituting these above estimates, choosing $\eta$ small enough, then we get
\ben\label{eq44}
\begin{split}
\f{d}{dt}&\int \Big[\f12|\na\theta|^2- \theta\big(2\mu|D u|^2+\lam(\div u)^2\big)\Big]\,dx+\int \rho\theta_t^2\,dx\\
&\leq \eta\|\na \dot u\|_{L^2}^2+C(\|u\cdot\na\theta\|_{L^2}^2+\|\na \theta\|_{L^2}^2+\|\na u\|_{L^2}^2+\|\na u\|_{L^4}^4),
\end{split}
\een
where $\eta$ is a small constant.

{\it Step 3: Closing the energy estimates.} Combining \eqref{eq43} and \eqref{eq44}, and choosing $\eta$ small enough, we get that
\ben\label{eq45}
\begin{split}
&\f{d}{dt}\int \Big[\rho|\dot{u}|^2+C|\na\theta|^2- \theta\big(2\mu|D u|^2+\lam(\div u)^2\big)\Big]\,dx+\int \Bigl(\mu|\nabla\dot{u}|^2+(\lambda+\mu) |\div\dot{u}|^2+C\rho\theta_t^2\Bigr)\,dx\\
&\leq C(\|\na u\|_{L^2}^2+\|\na u\|_{L^4}^4+\|\na \theta\|_{L^2}^2+\|u\|_{L^\infty}^2\|\na\theta\|_{L^2}^2).
\end{split}
\een
To conclude the estimate by Gronwall's inequality, we use the term $\|\sqrt{\rho}\dot{u}\|_{L^2 }$ to control $\|\nabla u\|_{L^4 }$. By Proposition \ref{energyprop} and \eqref{enep3}, we have
\beno
&&\|\na u\|_{L^\infty(0,\infty;L^2)}+\|P-1\|_{L^\infty(0,\infty;L^6)}\leq C, \\
&& \|\curl u\|_{L^6}+  \|G \|_{L^6}\leq \|\rho\dot{u}\|_{L^2}\leq C  \|\sqrt{\rho}\dot{u}\|_{L^2},
\eeno
which together with $\eqref{controlofu1}_2$ imply that
\begin{eqnarray*}
\|\nabla u\|^4_{L^4 }&\leq&\|\nabla u\|_{L^2 }\|\nabla
u\|^3_{L^6 }\leq C\|\nabla u\|_{L^6 }^2\|\nabla
u\|_{L^6 }\\
&\leq& C\|\nabla u\|^2_{L^6 }\big(\|\curl u\|_{L^6 }+\|G\|_{L^6 }+\|P-1\|_{L^6}\big)\\
&\leq& C\|\nabla u\|^2_{L^6 }\Big(1+\|\sqrt{\rho}\dot{u}\|_{L^2 }\Big)\le C\|\nabla u\|^2_{L^6 }\Big(1+\|\sqrt{\rho}\dot{u}\|_{L^2 }^2\Big).
\end{eqnarray*}

Substituting this estimate into \eqref{eq45} and noting that  $\|\nabla
u(t)\|^2_{L^6 }, \|u(t)\|^2_{L^\infty} \in L^1(0,\infty),$ by Proposition \ref{energyprop}, we get by
Gronwall's inequality that
\begin{eqnarray}\label{eq:w-high}
\int (\rho|\dot{u}|^2+C|\na\theta|^2)\,dx+\int^\infty_0\int ( |\nabla
\dot{u}|^2+C\rho\theta_t^2)\,dxdt\leq C,
\end{eqnarray}
with $C$ depending only on $\|\rho\|_{L^\infty}$ and $\rho_0,u_0,T_0$. Moreover, back to $\eqref{fcns}_3,$ we have $\Delta \theta\in L^2((0,\infty);L^2).$

By using \eqref{eq:w-high}, \eqref{eq45} can be improved as
\beno
&&\f{d}{dt}\int (\rho|\dot{u}|^2+C|\na\theta|^2)\,dx+\int ( |\nabla
\dot{u}|^2+C\rho\theta_t^2)\,dx\leq
C(\|\na u\|_{L^6}^2+\|\nabla {u}\|^2_{L^2}+\|u\|_{L^\infty}^2),\eeno
from which together with  \eqref{energyest1}, \eqref{eq:BD} and Sobolev imbedding theorem implies \eqref{energyest4}.

\end{proof}

\medskip

\subsubsection{Estimate for the propagation of $\na a$}  In this subsection, we want to give the proof to the upper bound of $\|\na u\|_{L^2((0,+\infty);L^\infty) }$ which in turn gives the estimates for propagation of $\na a$.   Here is the main result of this subsection:
\begin{prop}\label{prop_density}
Let $0<\alpha<1,$ $\mu>\frac12 \lam$ and $(\rho,u,T)$ be a global and smooth solution of \eqref{fcns} and satisfy \eqref{Ass1}--\eqref{Ass2}.  Initial data $(\rho_0,u_0,T_0)$ verify the admissible condition \eqref{admissible}.
Then
\ben\label{eq:blow}
 \| a\|_{L^\infty((0,+\infty);W^{1,6})\cap L^2((0,+\infty);W^{1,6})}+\|\na u\|_{L^2((0,+\infty);L^\infty) }\leq C,
\een
where $C$ depends on the initial data $(\rho_0, u_0, T_0)$ and $M$.  Moreover,
 \ben\label{L2naa}
	\f{d}{dt}\|\na a\|_{L^2}^2+\f{1}{4(2\mu+\lam)}\|  \na a\|_{L^2}^2  \leq
  C( \|\na \dot{u}\|_{L^2}^2+\|\dot{u}\|^2_{L^2}+\|\na\theta\|_{L^2}^2+\|\theta\|_{\dot{H}^2}^2+\|a\|_{L^6}^2).
  \een

 \end{prop}

\begin{proof}
First, we recall that an homogeneous Littlewood-Paley decomposition
$(\ddj)_{j\in\Z}$ is a dyadic decomposition in the Fourier space for $\R^d.$
One may for instance set $\ddj:=\varphi(2^{-j}D)$ with $\varphi(\xi):=\chi(\xi/2)-\chi(\xi),$
and $\chi$ a non-increasing nonnegative smooth function supported in $B(0,4/3),$ and with
value $1$ on $B(0,3/4)$  (see \cite{Bah}, Chap. 2 for more details). Then we have the interpolation inequality
\beno
\| \na \Lambda^{-1}   a\|_{L^\infty}&\leq & 2^{\f N2} \|a\|_{L^6}+\sum_{j\geq N}2^{-j \alpha  } (2^{j \alpha  }\|\ddj  a\|_{L^\infty})\le 2^{\f N2} \|a\|_{L^6}+ 2^{-N\alpha  }  \| a\|_{C^{\alpha}}.
\eeno
Choosing $2^{N(\f12+\alpha )}=\f{ \| a\|_{C^{\alpha}} }{\|a\|_{L^6}}$, we get that
$\| \na \Lambda^{-1}  a\|_{L^\infty}\leq C\|a\|_{L^6}^\beta \|a\|_{C^{\alpha}}^{1-\beta}$
with $\beta=1-\f{1}{1+2\alpha}\in (0,1).$ From Proposition \ref{energyprop}, we have $\theta\in L^2((0,\infty);\dot{H}^2)\subset L^2(C^{\f12}).$ Then similar to $a,$ we can get that $\| \na \Lambda^{-1}  \theta\|_{L^\infty}\leq C\|\theta\|_{L^6}^{\f12} \|\theta\|_{C^{\f12}}^{\f12}\leq C\|\na\theta\|_{L^2}^{\f12} \|\theta\|_{\dot{H}^2}^{\f12}.$ Therefore, we obtain that
\ben\label{nau}
\begin{split}
\|\nabla u\|_{L^\infty}&\leq C(\|\na\Lambda^{-1}\curl u\|_{L^\infty}+\| \na\Lambda^{-1}G\|_{L^\infty}+\| \na \Lambda^{-1} (P-1)\|_{L^\infty})\\
&\leq C(\|\curl u\|_{W^{1,6}}+\| G\|_{W^{1,6}}+\|a\|_{L^6}^\beta \|a\|_{C^{\alpha}}^{1-\beta}+\|\na\theta\|_{L^2}^{\f12} \|\theta\|_{\dot{H}^2}^{\f12}).
\end{split}
 \een

On the other hand, recall that $\div u=\f1{2\mu+\lam}(G+\rho T-1)=\f1{2\mu+\lam}(G+\theta+aT),$ it is not difficult to derive that
\ben\label{eq52}
&&\partial_t \na a+(u\cdot\nabla)\na a+\nabla u\nabla a+ \div u \nabla a  =-\rho \na\div u=-\f1{2\mu+\lam}\rho\bigl(\na(G+\theta)+T\na a+a\na\theta\bigr).
\een
 Multiplying  \eqref{eq52} by $ |\nabla a|^{p-2} \nabla a $
and integrating the resulting equation on $\R^3$, we can derive that for $p\ge 2$,
\beno
&&\f1p\f{d}{dt} \|\nabla a\|_{L^p}^p  +\f{1}{2\mu+\lam}\int \rho T |\nabla a|^p\,dx \\
&\leq&  \f1p\int \div u |\na a|^p\,dx+C\|\na u\|_{L^\infty}\|\nabla a\|_{L^p}^p+C (\|\na G\|_{L^p} +\|\na\theta\|_{L^p}) \|\na a\|_{L^p}^{p-1}\\
&\leq& \f1{(2\mu+\lam)p}\int (\rho T-1)|\na a|^p\,dx +C(\|G\|_{L^\infty}+\|\na u\|_{L^\infty})\|\nabla a\|_{L^p}^p+C (\|\na G\|_{L^p} +\|\na\theta\|_{L^p}) \|\na a\|_{L^p}^{p-1},
\eeno
which means that
\beno
&&\f1p\f{d}{dt} \|\nabla a\|_{L^p}^p  +\f{1}{(2\mu+\lam)p}\int[(p-1) \rho T+1] |\nabla a|^p\,dx \\
&\leq &C(\|G\|_{L^\infty}+\|\na u\|_{L^\infty})\|\nabla a\|_{L^p}^p+C (\|\na G\|_{L^p} +\|\na\theta\|_{L^p}) \|\na a\|_{L^p}^{p-1}.
\eeno
Noting that $p\geq 2$ and the non-negativity of density and temperature, dividing by $\|\na a\|_{L^p}^{p-2},$ we can obtain that
\beno
\begin{split}
 \f12\f{d}{dt} \|\nabla a\|_{L^p}^2  +\f{1}{(2\mu+\lam)p}\| \na a  \|_{L^p}^2 &\leq C(\|G\|_{L^\infty}+\|\na u\|_{L^\infty})\|\nabla  a\|_{L^p}^2+C (\|\na G\|_{L^p} +\|\na\theta\|_{L^p}) \|\na a\|_{L^p}\\
 &\leq\eta\|\na a\|_{L^p}^2+C\left((\|G\|_{W^{1,6}}^2+\|\na u\|_{L^\infty}^2)\|\na a\|_{L^p}^2+\|\na G\|_{L^p}^2+\|\na \theta\|_{L^p}^2\right),
\end{split}
\eeno
with $\eta$ small enough, which implies that
\ben\label{lpa}
\begin{split}
 \f{d}{dt} \|\nabla a\|_{L^p}^2  +\f{1}{(2\mu+\lam)p}\| \nabla a\|_{L^p}^2 \leq C\left((\|G\|_{W^{1,6}}^2+\|\na u\|_{L^\infty}^2)\|\na a\|_{L^p}^2+\|\na G\|_{L^p}^2+\|\na \theta\|_{L^p}^2\right).
 \end{split}
\een
By taking $p=6$ in \eqref{lpa} and using \eqref{nau} and \eqref{energyest4}, we obtain from Gronwall's inequality that
 $\|a\|_{L^\infty((0,+\infty);W^{1,6})\cap L^2((0,+\infty);W^{1,6})}\leq C.$ From which together with \eqref{nau} implies that $\|\na u\|_{L^2((0,+\infty);L^\infty) }\leq C.$ It completes the proof to \eqref{eq:blow}.

Now we go back to \eqref{lpa} with $p=2$. By Gronwall's inequality, we obtain that  $\na a\in L^\infty((0,+\infty);L^2)\cap L^2((0,+\infty);L^2).$ Thanks to the uniform-in-time bounds obtained in the above, \eqref{lpa} with $p=2$ will yield that
\beno \f{d}{dt} \|\nabla a\|_{L^2}^2  +\f{1}{4(2\mu+\lam)}\| \nabla a \|_{L^2}^2 \leq C( \|\nabla G\|_{L^{2}}^2+\|\nabla \theta\|_{L^{2}}^2+\|\theta\|_{\dot{H}^2}^2+\|G\|^2_{W^{1,6}} +\|\curl u\|_{W^{1,6}}^2 )+C_\eta\|a\|_{L^6}^2,
\eeno
from which together with the elliptic estimate for \eqref{eq:BD}, we obtain \eqref{L2naa}.
\end{proof}

As a consequence, we can prove the propagation of the lower bounds of the density and the temperature.

\begin{col}\label{prop:lower bound}
Let $0<\alpha<1,$ $\mu>\frac12 \lam$ and $(\rho,u,T)$ be a global and smooth solution of \eqref{fcns} and satisfy \eqref{Ass1}--\eqref{Ass2}.  Initial data $(\rho_0,u_0,T_0)$ verify the admissible condition \eqref{admissible}.
Then, there exists two constants $\underline{\rho}=\underline{\rho}(c,M)>0$ and $\underline{T}=\underline{T}(c,M)>0$ such that for all $t\ge0$, $\rho(t,x)\ge \underline{\rho}$ and $T(t,x)\ge \underline{T}$.

\end{col}

\begin{proof}
Using the equation of density and the estimate \eqref{eq:blow}, we have
\beno
\rho(t,x)\geq \rho_0(x)\exp^{-\int_{0}^t \|\div u\|_{L^\infty}\,d\tau}\ge ce^{-Ct^{\f12}}.
\eeno
 On the other hand, thanks to \eqref{est:a}, we  derive that $\lim\limits_{t\rightarrow \infty} \|a(t)\|_{L^6}=0,$ from which together with upper bounds for $\rho$ in $C^\alpha$,  we derive that
$\lim\limits_{t\rightarrow \infty} \|a(t)\|_{L^\infty}=0$.
These two facts imply that there exists a constant $ \underline{\rho}=\underline{\rho}(c,M_1)>0$ such that for all $t\ge0,$ $\rho(t,x)\ge \underline{\rho}.$

\medskip

Next we  prove the lower bound of the temperature. First, by maximum principle of heat equation, it is not hard to prove that the temperature is also positive. Next, let $f=\f{1}{T}-1$ which satisfies that
\beno
\rho(\pa_t +u\cdot \na)f-\Delta f &=& - (f+1)^2 \Big((\f{\mu}2|\na u+(\na u)'|^2+\lambda(\div u)^2)+T^{-1}|\na T|^2\Big)+ \rho (f+1) \div u \\
&\leq & \rho (f+1) \div u.
\eeno
Define $f_{+}=\max\{f,0\}$. Multiplying $|f_{+}|^{p-2}f_{+}$ on the both sides of the above inequality to get
\beno
\f1p\f{d}{dt} \|\rho^{\f1p}f_{+}\|^p_{L^p}\leq   ( \| \rho   \div u\|_{L^\infty}\|f_{+}\|_{L^p}+\| \rho   \div u\|_{L^p})\|f_{+}\|^{p-1}_{L^p}.
\eeno
Recalling that   the density has the positive lower bound, we get there exists a constant $C$ independent of $p$ such that
\beno
 \f{d}{dt}  \|\rho^{\f1p}f_{+}\|_{L^p}\leq  C ( \| \rho   \div u\|_{L^\infty}\|f_{+}\|_{L^p}+\| \rho   \div u\|_{L^p}),
\eeno
which implies that for all $t>0$, it holds
\beno
\|f_{+}(t)\|_{L^p}\leq (1+\|f_{+}(0)\|_{L^p})e^{Nt},
\eeno
where $N= \| \rho\|_{L^\infty} ( \|\div u\|_{L^2}+\|   \div u\|_{L^\infty} )$. Let $p\to \infty$, we get that
\beno
\|f_{+}(t)\|_{L^\infty}\leq (1+\|f(0)\|_{L^2\cap L^\infty} )e^{Nt},
\eeno
This shows that in any finite time  the temperature has positive lower bound.

Secondly we will prove the  $\|\theta\|_{L^6}$ will converge to zero when   $t$ goes to infinity.
Multiplying $|\theta|^4\theta$ on the both sides of \eqref{eqoftheta} and integrating by part, we  obtain that
\beno
\f{d}{dt}  \|\rho^\f16 \theta\|^6_{L^6}+\|\theta^2 \na \theta\|^2_{L^2}\leq \|\na u\|^2_{L^2} +\|\na \theta\|^2_{L^2}+\|  \na a\|^2_{L^2}.
\eeno
Because that the right hand of the above inequality is integrable with respect to time, we conclude that
$\|\rho^\f16 \theta\|^6_{L^6}$ is convergent when $t$ goes to infinity. Recalling that $\theta\in L^2([0,\infty];L^6)$, we get the desired result.

 To conclude our result, by the   argument used for the density, it seems that we only need to get the uniform-in-time bounds for the high order estimate of $\theta$.
Taking $\pa_t$ on the both sides of $\eqref{fcns}_3$, we get
\beno
\rho\pa_t \theta_t-\Delta \theta_t= -a_t\theta_t-(\rho u\cdot \na u+\rho T\div u)_t+ (\f{\mu}2|\na u+(\na u)'|^2+\lambda(\div u)^2)_t.
\eeno
Then by energy estimates, we have
\beno
\f{d}{dt} \|\rho^\f12 \theta_t\|^2_{L^2}+\|\na\theta_t\|^2_{L^2}\leq C(\|\theta_t\|^2_{L^2}+\|\na a\|^2_{L^2\cap L^6}+\|\na u\|^2_{H^1}+\|\na u\|^2_{L^\infty}+\|  u_t\|^2_{H^1})(\|  \theta_t\|^2_{L^2}+1).
\eeno
Applying the Gronwall's inequality and the results obtained in the previous steps, it is easy to get that
\beno
\|\theta_t\|_{L^\infty(0,\infty; L^2)}<\infty.
\eeno
 Going back to the $\eqref{fcns}_3$, we obtain that
$\|\na^2\theta\|_{L^\infty(0,\infty; L^2)}<\infty$
which implies that Holder continuous of the temperature. By copying the  argument used for the density, the above facts will imply that there exists a constant $ \underline{T}=\underline{T}(c,M)>0$ such that for all $t\ge0,$ $T(t,x)\ge \underline{T}.$

\end{proof}

 \subsubsection{Deriving the dissipation inequality} We want to prove
\begin{prop}\label{energy_prop_main}
Let $0<\alpha<1,$ $\mu>\frac12 \lam,$ and $(\rho,u,T)$ be a global and smooth solution of \eqref{fcns} and satisfy \eqref{Ass1}--\eqref{Ass2}. Initial data $(\rho_0,u_0,T_0)$ verify the admissible condition \eqref{admissible}. Then there exist
constants $A_i(i=1,\cdots,7)$ are positive constants depending on $\mu,$ $\lam$ and $M$ such that
\beno
\begin{split} X(t)&\eqdefa   A_1\|\rho^{\f14} u\|_{L^4}^4+A_2\Big( \mu\|\na u\|_{L^2}^2+(\lam+\mu)\|\div u\|_{L^2}^2-\int (P-1)\div u\,dx- \f1{2\mu+\lam}\int  \bigl(H(a,\theta)-F(a)\bigr)\,dx\Big)\\
&+A_3\|  a \|_{L^6}^2+A_4\big(\int (\rho\ln\rho-\rho+1)\,dx+\|\sqrt{\rho} u\|_{L^2}^2+\int\rho(T-\ln T-1)\,dx\big)\\
&+A_5(\|\sqrt{\rho}\dot{u}\|_{L^2}^2+\|\na\theta\|_{L^2}^2)+A_6\|\na a\|_{L^2}^2\\
&\sim \|    u\|_{H^1}^2+\| a\|_{H^1}^2 +\| \dot u\|_{L^2}^2+\|\theta\|_{H^1}^2, \end{split}\eeno
which verifies
 \begin{equation}\label{energyest1_main}
\begin{split}
\frac{d}{dt} X(t)
+A_7\Big( \|\na^2 u\|_{L^2}^2+\|\na   u\|_{L^2}^2+\| \na a\|_{L^2}^2 +\|\na\dot u\|_{L^2}^2+\|\na^2\theta\|_{L^2}^2\Big)\leq 0.
\end{split}
\end{equation}
\end{prop}
\begin{proof} From \eqref{L2naa} and \eqref{energyest4}, we get that there exist constants $A_i(i=1,\cdots,7)$ such that
\beno
\begin{split}
&\f{d}{dt}\Bigl[A_1\|\rho^{\f14} u\|_{L^4}^4+A_2\Big( \mu\|\na u\|_{L^2}^2+(\lam+\mu)\|\div u\|_{L^2}^2-\int (P-1)\div u\,dx- \f1{2\mu+\lam}\int  \bigl(H(a,\theta)-F(a)\bigr)\,dx\Big)\\
&\quad+A_3\|  a \|_{L^6}^2+A_4\Big(\int (\rho\ln\rho-\rho+1)\,dx+\|\sqrt{\rho} u\|_{L^2}^2+\int\rho(T-\ln T-1)\,dx\Big)\\
&\quad+A_5(\|\sqrt{\rho}\dot{u}\|_{L^2}^2+\|\na\theta\|_{L^2}^2)+A_6\|\na a\|_{L^2}^2\Bigr]\\
&\qquad+A_7\Big( \|\na^2 u\|_{L^2}^2+\|\na   u\|_{L^2}^2+\| \na a\|_{L^2}^2 +\|\na\dot u\|_{L^2}^2+\|\na^2\theta\|_{L^2}^2\Big)\leq 0.
\end{split}
\eeno	

  Thanks to the energy identity \eqref{energyest2}, the constant $A_4$ can be chosen large enough to ensure that $X(t)\ge0$. Thanks to  the positive lower and upper bounds  of density and temperature, one has $\int (\rho\ln\rho-\rho+1)\,dx\sim \|\rho-1\|_{L^2}^2$
  and $\int\rho(T-\ln T-1)\,dx\sim \|T-1\|_{L^2}^2$, from which together with $\rho \dot u+\na P=\mu\Delta u+(\lam+\mu)\na\div u$, we deduce that   $X(t)\sim \|u\|_{H^1}^2+\| a\|_{H^1}^2 +\| \dot u\|_{L^2}^2+\|\theta\|_{H^1}^2$.  It ends the proof of the proposition.
\end{proof}

  \subsection{Time-frequency splitting method and the longtime behavior of the solution}
The aim of this subsection is to show the longtime behavior of the solution with quantitative estimates.

First, note that $\rho_t=a_t$ and $T_t=\theta_t,$ we can make a small modification on \eqref{cns} to get
\begin{equation}\label{cns1}
 \left\{\begin{array}{l}
\partial_t a+\div(\rho  u )=0,\\[0.5ex]\displaystyle
\partial_t(\rho  u)+\div (\rho  u \otimes u )
+\nabla (P-1)=\div \bS (u),\\
\pa_t(\rho E_1)+\div(\rho E_1u)+\div(Pu)=\div\bigl(\bS(u)u\bigr)+\Delta \theta,
\end{array}
\right.
\end{equation}
where $\rho E_1=\rho \theta+\f12\rho|u|^2.$ We begin with a crucial lemma on the estimate of the low frequency part of the solution.

\begin{lem}\label{conlf}  Let $a_0,u_0,\theta_0\in L^1(\R^3)\cap L^2(\R^3),$ and $\rho_0\in L^\infty(\R^3).$ Then if $\rho(t,x)\le M$, we have
\begin{equation}\label{conlfsol}
\begin{split}
 &\int_{S(t)} \big( |\hat{a}(\xi,t)|^2+|\widehat{\rho u}(\xi,t)|^2+|\widehat{\rho E_1}(\xi,t)|^2 \big) \,d\xi \leq
  C\big(\|a_0\|_{L^1}^2+ \|\rho_0 u_0\|_{L^1}^2 +\|\rho_0\theta_0\|_{L^1}^2+\|\rho_0|u_0|^2\|_{L^1}^2 \big)(1+t)^{-\f32}\\
	&\qquad+ C(M)(1+t)^{-\f32}\int_0^t \big( \|u\|_{L^2}^4+ \|a\|_{L^2}^4+\|\theta\|_{L^2}^4+\|\na u\|_{L^2}^4+\|u\|_{L^2}^3\|\na u\|_{L^2}^3\big)\,ds,  \end{split}
\end{equation}
where $S(t) =\{\xi\in \R^3: |\xi|\leq C(1+t)^{-\f12}  \}.$
\end{lem}

\begin{proof}  We take the Fourier transform of \eqref{cns1}, and then multiply $\bar{\hat{a}}$ to the first equation, multiply $\overline{\widehat{\rho u}}$ to the second equation, multiply $\overline{\widehat{\rho E_1}}$ to the third equation respectively to obtain that
	\beno
	\left\{
	\begin{aligned}
		& \f12\f{d}{dt} |\hat{a}|^2 + i \xi \cdot \widehat{\rho u} \bar{\hat{a}}=0,\\
		&\f12\f{d}{dt} |\widehat{\rho u}|^2+\big(\widehat{\div(\rho u\otimes u)}-\mu\widehat{\Delta u}-(\lam+\mu)\widehat{\na\div u} \big) \cdot \overline{\widehat{\rho u}}+ i  \xi(\widehat{a \theta}+\hat{a}+\hat{\theta} ) \cdot \overline{\widehat{\rho u}}=0,\\
		&\f12\f{d}{dt}|\widehat{\rho E_1}|^2+\big(\widehat{\div(\rho E_1 u)}+\widehat{\div((P-1)u)}+i\xi\cdot \hat{u} \big)  \overline{\widehat{\rho E_1}}=\big(i\xi\cdot\widehat{\bS(u)u}+\widehat{\Delta \theta}\big)\overline{\widehat{\rho E_1}},
	\end{aligned}
	\right.
	\eeno
	which implies that
	\beno
	\begin{split}
	\f12 \f{d}{dt} \big(|\hat{a}|^2+  |\widehat{\rho u}|^2+|\widehat{\rho E_1}|^2\big)&=\big[ -\widehat{\div(\rho u\otimes u)}+\mu\widehat{\Delta u}+(\lam+\mu)\widehat{\na\div u}+ i  \xi\widehat{a\theta} \big]\cdot\overline{\widehat{\rho u}}-i\xi\hat{\theta}\cdot\overline{\widehat{au}}\\
	&-  \big[\widehat{\div(\rho E_1 u)}+\widehat{\div((P-1)u)}-i\xi\cdot\widehat{\bS(u)u}- \widehat{\Delta \theta} \big]\overline{\widehat{\rho E_1}} -i\xi\cdot \hat{u} \big(\overline{\widehat{a\theta}}+ \overline{\widehat{\f12\rho|u|^2}}\big)\\
	& \overset{\text{def}} {=}F(\xi,t).
	\end{split}
	\eeno
	Integrating the above equation with time $t$, we get that
	\beno
	|\hat{a}(\xi,t)|^2+|\widehat{\rho u}(\xi,t)|^2+|\widehat{\rho E_1}(\xi,t)|^2=|\hat{a}(\xi,0)|^2+|\widehat{\rho u}(\xi,0)|^2 +|\widehat{\rho E_1}(\xi,0)|^2+2\int_0^t F(\xi, s)\,ds.
	\eeno
	Let $S(t)  \overset{\text{def}}{=}\{\xi: |\xi|\leq C(1+t)^{-\f12}  \},$ then we can split the phase space $\R^3$ into two time-dependent regions, $S(t)$ and $S(t)^c.$ Integrating the above equation over $S(t),$ and noting that $\widehat{\rho u}=\hat{u}+\widehat{au},$ $\widehat{\rho \theta}=\hat{\theta}+\widehat{a\theta},$ and
	\beno
	\widehat{\Delta u}\bar{\hat{u}}=-|\xi|^2 |\hat{u}|^2, \quad \widehat{\na\div u} \bar{\hat{u}}= -|\xi \cdot \hat{u}|^2, \quad \widehat{\Delta \theta}\bar{\hat{\theta}}=-|\xi|^2 |\hat{\theta}|^2,
	\eeno
	we can obtain that
	\begin{equation}\label{decayeq1}
	\begin{split}
	 &\int_{S(t)} \big(|\hat{a}(\xi,t)|^2+|\widehat{\rho u}(\xi,t)|^2+|\widehat{\rho E_1}(\xi,t)|^2 \big) \,d\xi +\int_0^t \int_{S(t)} \big( \mu |\xi|^2 |\hat{u}|^2 + (\lam+\mu) |\xi\cdot \hat{u}|^2+|\xi|^2 |\hat{\theta}|^2 \big) \,d\xi ds\\
	&=\int_{S(t)} \big( |\hat{a}(\xi,0)|^2+|\widehat{\rho u}(\xi,0)|^2+|\widehat{\rho E_1}(\xi,0)|^2 \big) \,d\xi \\
	&\qquad+ \int_0^t \int_{S(t)} \Big[ -\widehat{\div(\rho u\otimes u)} \cdot\overline{\widehat{\rho u}} +\Big(\mu\widehat{\Delta u}  +(\lam+\mu)\widehat{\na\div u}\Big) \cdot\overline{\widehat{a u}} +   i  \xi\big(\widehat{a\theta}\cdot\overline{\widehat{\rho u}}-\hat{\theta}\cdot\overline{\widehat{au}}\big)\\
	&\qquad -\Big( \widehat{\div(\rho E_1 u)}+\widehat{\div((P-1)u)}-i\xi\cdot\widehat{\bS(u)u}  \Big)\overline{\widehat{\rho E_1}} -\big(|\xi|^2\hat{\theta}+i\xi\cdot\hat{u}\big)\big(\overline{\widehat{a\theta}}+ \overline{\widehat{\f12\rho|u|^2}}\big)\Big]  \,d\xi ds \\
	&\eqdefa \int_{S(t)} \big( |\hat{a}(\xi,0)|^2+|\widehat{\rho u}(\xi,0)|^2+|\widehat{\rho E_1}(\xi,0)|^2 \big) \,d\xi + \sum_{i=1}^5 B_i.
	\end{split}
\end{equation}
	From Proposition \ref{energyprop}, we have that $a,$ $u,$ $\theta$ and $\na u$ all belong to $L^\infty((0,+\infty);L^2),$ which means $\widehat{\rho u\otimes u},$ $\widehat{\f12\rho|u|^2},$ $\widehat{a u},$ $\widehat{\rho\theta u},$ $\widehat{a\theta},$ $\widehat{(P-1)u}$ and $\widehat{\bS(u)u}$ belong to $L^\infty((0,+\infty);L^\infty).$ Thanks to these facts, we can give estimates to the terms $B_i(i=1,\cdots,5).$ We first have
	\ben\label{decayeq2}
	\begin{split}
	|B_1|&\leq \Big| \int_0^t \int_{S(t)}  \widehat{\div(\rho u\otimes u)}\cdot (\overline{\hat{u}}+\overline{\widehat{a u}})\,d\xi ds\Big|\\
	&\leq \eta \int_0^t \int_{S(t)}  \mu |\xi|^2 |\hat{u}|^2 \,d\xi ds  +C_{\eta} \int_0^t \int_{S(t)} \big| \widehat{\rho u\otimes u} \big|^2 \,d\xi ds +\int_0^t \int_{S(t)} |\xi| \big| \widehat{\rho u\otimes u} \big| |\widehat{a u}| \,d\xi ds\\
	 &\leq \eta \int_0^t \int_{S(t)}  \mu |\xi|^2 |\hat{u}|^2 \,d\xi ds +C_{\eta} \int_0^t \|\widehat{\rho u\otimes u}\|_{L^\infty}^2 \int_{S(t)}\,d\xi ds \\
	 &\qquad + C(1+t)^{-\f12}\int_0^t \|\widehat{\rho u\otimes u}\|_{L^\infty} \|\widehat{a u}\|_{L^\infty}\int_{S(t)}\,d\xi ds \\
	 &\leq \eta \int_0^t \int_{S(t)}  \mu |\xi|^2 |\hat{u}|^2 \,d\xi ds+C_{\eta}(1+t)^{-\f32}\int_0^t \|u\|_{L^2}^4 ds + C(1+t)^{-2}\int_0^t \|u\|_{L^2}^3 \|a\|_{L^2}\,ds.
	\end{split}
\een
	Similarly, one has
	\begin{equation}\label{decayeq3}
	\begin{split}
	|B_2|&\leq \eta \int_0^t \int_{S(t)}  \mu |\xi|^2 |\hat{u}|^2 \,d\xi ds +C_{\eta} \int_0^t \int_{S(t)} |\xi|^2 \big| \widehat{a u} \big|^2 \,d\xi ds \\
	&\leq \eta \int_0^t \int_{S(t)}  \mu |\xi|^2 |\hat{u}|^2 \,d\xi ds +C_{\eta}(1+t)^{-\f52}\int_0^t \|u\|_{L^2}^2 \|a\|_{L^2}^2\,ds,
	\end{split}
	\end{equation}
	and
	\begin{equation}\label{decayeq4}
	\begin{split}
	|B_3|&\leq \eta \int_0^t \int_{S(t)}  (\mu |\xi|^2 |\hat{u}|^2+ |\xi|^2 |\hat{\theta}|^2)\,d\xi ds +C_{\eta} \int_0^t \int_{S(t)} \big( \big| \widehat{a u} \big|^2+\big| \widehat{a \theta} \big|^2\big) \,d\xi ds  + \int_0^t \int_{S(t)}  |\xi| | \widehat{a\theta}|  | \widehat{au}|  \,d\xi ds \\
	&\leq \eta \int_0^t \int_{S(t)}  (\mu |\xi|^2 |\hat{u}|^2+ |\xi|^2 |\hat{\theta}|^2) \,d\xi ds +C_{\eta}(1+t)^{-\f32}\int_0^t (\|u\|_{L^2}^2+\|\theta\|_{L^2}^2) \|a\|_{L^2}^2\,ds\\
	&\quad +C(1+t)^{-2} \int_0^t   \|u\|_{L^2}\|\theta\|_{L^2}\|a\|_{L^2}^2  \, ds.
	\end{split}
	\end{equation}

For $B_4,$ recalling that $\rho E_1=\theta+a\theta+\f12\rho|u|^2$ and $\big\||u|^3\big\|_{L^1}\leq \|u\|_{L^2}^{\f32}\|\na u\|_{L^2}^{\f32},$ we can obtain that
\begin{equation}\label{decayeq5}
\begin{split}
|B_4|&\leq \eta \int_0^t \int_{S(t)}  |\xi|^2 |\hat{\theta}|^2\,d\xi ds +C_{\eta}\int_0^t \int_{S(t)} \big(\big| \widehat{\rho E_1 u} \big|^2 +\big| \widehat{(P-1) u} \big|^2+\big| \widehat{\bS(u)u} \big|^2\big)\,d\xi ds\\
&\quad +\int_0^t \int_{S(t)} |\xi| \big(\big| \widehat{\rho E_1 u} \big| +\big| \widehat{(P-1) u} \big|+\big| \widehat{\bS(u)u} \big|\big) \big( \big| \widehat{a\theta} \big| +\big| \widehat{\f12\rho|u|^2} \big|  \big)\,d\xi ds\\
&\leq \eta \int_0^t \int_{S(t)}  |\xi|^2 |\hat{\theta}|^2\,d\xi ds +C_{\eta}(1+t)^{-\f32}\int_0^t [\|u\|_{L^2}^2(\|\theta\|_{L^2}^2+\|a\|_{L^2}^2+\|\na u\|_{L^2}^2)+\|u\|_{L^2}^3\|\na u\|_{L^2}^3] \,ds\\
&\quad+C(1+t)^{-2}\int_0^t \Big[\|u\|_{L^2}(\|\theta\|_{L^2}+\|a\|_{L^2}+\|\na u\|_{L^2})+\|u\|_{L^2}^{\f32}\|\na u\|_{L^2}^{\f32} \Big] (\|a\|_{L^2}\|\theta\|_{L^2}+\|u\|_{L^2}^2)\,ds.
\end{split}
\end{equation}	
Similar to $B_2$ and $B_4,$ one has
\begin{equation}\label{decayeq6}
\begin{split}
|B_5|&\leq \eta \int_0^t \int_{S(t)}  (\mu |\xi|^2 |\hat{u}|^2+ |\xi|^2 |\hat{\theta}|^2)\,d\xi ds +C_{\eta}\int_0^t \int_{S(t)} |\xi|^2\big(\big| \widehat{a\theta} \big|^2 +\big| \widehat{\f12\rho|u|^2} \big|^2\big)\,d\xi ds\\
&\quad +C_{\eta}\int_0^t \int_{S(t)} \big(\big| \widehat{a\theta} \big|^2 +\big| \widehat{\f12\rho|u|^2} \big|^2\big)\,d\xi ds\\
&\leq \eta \int_0^t \int_{S(t)}  (\mu |\xi|^2 |\hat{u}|^2+ |\xi|^2 |\hat{\theta}|^2)\,d\xi ds +C_{\eta}(1+t)^{-\f52}\int_0^t (\|a\|_{L^2}^2\|\theta\|_{L^2}^2+\|u\|_{L^2}^4)\,ds\\
&\quad+C(1+t)^{-\f32}\int_0^t (\|a\|_{L^2}^2\|\theta\|_{L^2}^2+\|u\|_{L^2}^4)\,ds.
\end{split}
\end{equation}

	Note that $a_0,$ $\rho_0 u_0,$ $\rho_0 \theta_0$ and $\rho_0 |u_0|^2$ belong to $L^1(\R^3).$ Then we have
	\begin{equation}\label{decayeq7}
	\begin{split}
	&\int_{S(t)} \big( |\hat{a}(\xi,0)|^2+|\widehat{\rho u}(\xi,0)|^2+|\widehat{\rho E_1}(\xi,0)|^2 \big) \,d\xi\\
	&\leq C\big(\|a_0\|_{L^1}^2+ \|\rho_0 u_0\|_{L^1}^2 +\|\rho_0\theta_0\|_{L^1}^2+\big\|\rho_0|u_0|^2\big\|_{L^1}^2 \big)(1+t)^{-\f32}.
	\end{split}
	\end{equation}
	Plugging $\eqref{decayeq2}-\eqref{decayeq7}$ into \eqref{decayeq1}, and choosing $\eta$ small enough, we arrive at
	\begin{equation*}
	\begin{split}
	&\int_{S(t)} \big( |\hat{a}(\xi,t)|^2+|\widehat{\rho u}(\xi,t)|^2+|\widehat{\rho E_1}(\xi,t)|^2 \big) \,d\xi +C\int_0^t \int_{S(t)} \big( \mu |\xi|^2 |\hat{u}|^2 + (\lam+\mu) |\xi\cdot \hat{u}|^2+|\xi|^2 |\hat{\theta}|^2 \big) \,d\xi ds\\
	&\leq C\big(\|a_0\|_{L^1}^2+ \|\rho_0 u_0\|_{L^1}^2 +\|\rho_0\theta_0\|_{L^1}^2+\big\|\rho_0|u_0|^2\big\|_{L^1}^2 \big)(1+t)^{-\f32}\\
	&\quad+ C(1+t)^{-\f32}\int_0^t \big( \|u\|_{L^2}^4+ \|a\|_{L^2}^4+\|\theta\|_{L^2}^4+\|\na u\|_{L^2}^4+\|u\|_{L^2}^3\|\na u\|_{L^2}^3\big)\,ds.
	\end{split}
	\end{equation*}
It ends the proof to the lemma.
\end{proof}

\medskip

Now we are in a position to prove

\begin{prop}\label{decayprop}  Let $(\rho,u,T)$ be a global smooth solution of \eqref{fcns} and satisfy \eqref{Ass1}--\eqref{Ass2} with  $0<\alpha<1$ and $\mu>\frac12 \lam$. Assume that initial data $(\rho_0,u_0,T_0)$ verify the admissible condition \eqref{admissible}. Suppose that $(a_0,\theta_0) \in L^1(\R^3)\cap H^1(\R^3)$ and $u_0 \in L^1(\R^3)\cap H^2(\R^3),$ then we have
\begin{equation}\label{decayest2}
\|u(t)\|_{H^1} + \|a(t)\|_{H^1}+ \|\theta(t)\|_{H^1}\leq C(1+t)^{-\f34},
\end{equation}
where the constant $C$ depends only on $\mu,$ $\lam,$ $M,$ $\|a_0\|_{L^1\cap H^1},$ $\|u_0\|_{L^1\cap H^2}$ and $\|\theta_0\|_{L^1\cap H^1}.$
\end{prop}

\begin{proof} We separate the proof into several steps.
	
	{\it Step 1: The first sight of the convergence.} Thanks to \eqref{conlfsol} and the fact that $a,$ $u,$ $\theta$ and $\na u$ belong to $L^\infty((0,+\infty);L^2),$ we have
\begin{equation}\label{decayeq9}
\begin{split}
&\int_{S(t)} \big( |\hat{a}(\xi,t)|^2+|\widehat{\rho u}(\xi,t)|^2+|\widehat{\rho E_1}(\xi,t)|^2 \big) \,d\xi \\
&\leq C\big(\|a_0\|_{L^1}^2+ \|\rho_0 u_0\|_{L^1}^2+\|\rho_0\theta_0\|_{L^1}^2+\big\|\rho_0|u_0|^2\big\|_{L^1}^2 \big)(1+t)^{-\f32} \\
&\quad+C \big(\|u\|_{L^\infty(L^2)}^4+\|a\|_{L^\infty(L^2)}^4+\|\theta\|_{L^\infty(L^2)}^4+\|\na u\|_{L^\infty(L^2)}^4+\|u\|_{L^\infty(L^2)}^3\|\na u\|_{L^\infty(L^2)}^3\big) (1+t)^{-\f12}\\
&\leq C(1+t)^{-\f12}.
\end{split}
\end{equation}
Due to the fact $u=\rho u- a u$ and $\theta=\rho E_1-a\theta-\f12\rho|u|^2,$ we have
\beno
\int_{S(t)}  |\widehat{  u}(\xi,t)|^2  \,d\xi &\leq& \int_{S(t)}   |\widehat{\rho u}(\xi,t)|^2  \,d\xi+\int_{S(t)} |\widehat{a u}(\xi,t)|^2   \,d\xi\\
&\leq&C(1+t)^{-\f12}+C(1+t)^{-\f32}|\widehat{a u}(\xi,t)|_{L^\infty}^2\\
&\leq&C(1+t)^{-\f12},
\eeno
and
\beno
\int_{S(t)}  |\widehat{ \theta}(\xi,t)|^2  \,d\xi &\leq& \int_{S(t)}   |\widehat{\rho E_1}(\xi,t)|^2  \,d\xi+\int_{S(t)} |\widehat{a \theta}(\xi,t)|^2   \,d\xi+\int_{S(t)} \big|\widehat{\f12\rho|u|^2}(\xi,t)\big|^2   \,d\xi\\
&\leq&C(1+t)^{-\f12}+C(1+t)^{-\f32}\big(|\widehat{a \theta}(\xi,t)|_{L^\infty}^2+\big|\widehat{\rho|u|^2}(\xi,t)\big|_{L^\infty}^2\big)\\
&\leq&C(1+t)^{-\f12}.
\eeno

Next, because of $\rho \dot u=\mu\Delta u+(\lam+\mu)\na\div u-\na P$, following the same argument, we obtain
\beno
\int_{S(t)}  |\widehat{  \rho \dot {u} }(\xi,t)|^2  \,d\xi \leq  \int_{S(t)}  |\mu\widehat{ \Delta u}+(\lam+\mu)\widehat{\na\div u}-  \big(\widehat{\na (a\theta)}+\widehat{\na a} +\widehat{\na \theta}\big)(\xi,t)|^2  \,d\xi  \leq  C(1+t)^{-\f32},
\eeno
which implies that
$
\int_{S(t)}  |\widehat{   \dot {u} }(\xi,t)|^2  \,d\xi   \leq  C(1+t)^{-\f32}.
$

 We recall the dissipation inequality \eqref{energyest1_main}. Then by time-frequency splitting method, it is not difficult to derive that
 \beno
\f{d}{dt} X(t)+ \f{K}{1+t} X(t)&\leq& \f{1}{1+t}\int_{S(t)}\big( |\hat{a}(\xi,t)|^2+|\widehat{  u}(\xi,t)|^2+|\widehat{ \dot{u}}(\xi,t)|^2 \big) \,d\xi\\&\le&
  C(1+t)^{-\f32},
\eeno
which implies
\begin{equation}\label{informdecay}
X(t)\leq C(1+t)^{-\f12}.
\end{equation}
In particular, we have
\begin{equation}\label{decayeq10}
\|u\|_{L^2}+ \|a\|_{L^2}+ \|\theta\|_{L^2}+ \|\na u\|_{L^2} \leq C(1+t)^{-\f14}.
\end{equation}

{\it Step 2: Improving the decay estimate (I).}
We want to improve the decay estimate. Thanks to \eqref{conlfsol} and \eqref{decayeq10}, we improve the estimate for the low frequency part as follows
\beno
\int_{S(t)} \big( |\hat{a}(\xi,t)|^2+|\widehat{\rho u}(\xi,t)|^2+|\widehat{\rho E_1}(\xi,t)|^2 \big) \,d\xi  \leq    C(1+t)^{-\f32}\log(1+t).
\eeno
Now following the simiar argument used in the previous step, we conclude that
\beno \int_{S(t)} \big( |\hat{a}(\xi,t)|^2+|\widehat{u}(\xi,t)|^2+|\widehat{\theta}(\xi,t)|^2+|\widehat{\dot{u}}(\xi,t)|^2 \big) \,d\xi  \leq  C(1+t)^{-\f32}\log(1+t), \eeno
which implies that
\beno
\f{d}{dt} X(t)+ \f{K}{1+t} X(t)\leq C(1+t)^{-1} (1+t)^{-\f32}\log(1+t).
\eeno
We obtain that
\begin{equation}\label{decayeq11}
X(t)\leq C(1+t)^{-\f32}\log(1+t).
\end{equation}

{\it Step 3:  Improving the decay estimate (II). } Finally we try to get the sharp decay estimate. By \eqref{decayeq11}, we have
$\|u\|_{L^2}+ \|a\|_{L^2}+ \|\theta\|_{L^2}+ \|\na u\|_{L^2}  \leq C(1+t)^{-\f12}.$ Now
we may repeat the same process in the above to get that
\beno \int_{S(t)} \big( |\hat{a}(\xi,t)|^2+|\widehat{u}(\xi,t)|^2+|\widehat{\theta}(\xi,t)|^2+|\widehat{\dot{u}}(\xi,t)|^2 \big) \,d\xi  \leq  C(1+t)^{-\f32}, \eeno
which implies that
\beno
\f{d}{dt} X(t)+ \f{K}{1+t} X(t)\leq C(1+t)^{-1}(1+t)^{-\f32}.
\eeno
It is enough to derive \eqref{decayest2}. We ends the proof to the proposition.
\end{proof}

\medskip

\subsection{Proof of Theorem \ref{thm:main1}} Finally, we are in the position to prove Theorem \ref{thm:main1}.
\begin{proof}[Proof of Theorem \ref{thm:main1}]
We first note that the second results of the theorem are proved by Proposition  \ref{decayprop}. Due to Proposition \ref{energy_prop_main}, we have $a\in L^\infty((0,+\infty);H^1)\cap L^2((0,+\infty); \dot{H}^1),$ $u\in L^\infty((0,+\infty);H^2),$ $\na u\in L^2((0,+\infty); H^1\cap L^\infty),$ $\theta\in L^\infty((0,+\infty);H^1)$ and $\na \theta\in L^2((0,+\infty); H^1).$ Then the desired result is reduced to the proof of the propagation of $\na^2 a$.
	
We notice that  by Proposition \ref{prop_density}, $\na a\in L^\infty((0,\infty); L^p)$ with $p\in[2,6]$, which will be used frequently in what follows. Recall that
\begin{equation}\label{aeq}
\na^2a_t +\na^2( u\cdot\na a)+\na^2 a\,\div u+\na a\na\div u+\rho\na^2\div u=0,
\end{equation}
and $\div u=\f1{2\mu+\lam}(G+a\theta +a+\theta),$ then it is not difficult to derive that
\beno
\begin{split}
\f1{2}\f{d}{dt} \|\na^2 a\|_{L^2}^2 &+ \f1{2\mu+\lam} \int \rho T|\na^2 a|^2\,dx \leq \Big| \int \div u |\na^2 a|^2\,dx\Big| + \Big|\int \na a\na\div u \na^2 a\,dx \Big|\\
&+\f1{2\mu+\lam}\Big|\int \rho(\na^2 G +\rho\na^2\theta +\na a\na\theta)\na^2 a\,dx \Big|+ \Big|\int \na^2 (u\cdot \na \mathfrak{a})\na^2 a\,dx \Big|\\
&\overset{\text{def}}{=} D_1+D_2+D_3+D_4.
\end{split}
\eeno
By Cauchy-Schwartz inequality and Proposition \ref{prop_density}, we can estimate $D_i(i=1,2)$ easily by
\beno
D_1&\leq& \|\div u\|_{L^\infty} \|\na^2 a\|_{L^2}^2 \leq  \eta \|\na^2 a\|_{L^2}^2+C_\eta \|\na u\|_{L^\infty}^2\|\na^2 a\|_{L^2}^2,\\
D_2&\leq& C  \|\na a\|_{L^3} \|\na^2 u\|_{L^6} \|\na^2 a\|_{L^2}\leq  \eta \|\na^2 a\|_{L^2}^2 + C_\eta \|\na a\|_{L^3}^2 \|\na^2 u\|_{L^6}^2.
\eeno
For $D_3$, we have
\beno
D_3&\leq& C ( \|\na^2 G \|_{L^2}+\|\na^2 \theta\|_{L^2}) \|\na^2 a\|_{L^2}+C\|\na a\|_{L^3} \|\na\theta \|_{L^6} \|\na^2 a\|_{L^2}\\
&\leq &\eta \|\na^2 a\|_{L^2}^2 +C_\eta (\|\na^2 G \|_{L^2}^2+\|\na^2 \theta\|_{L^2}^2 +\|\na a\|_{L^3}^2 \|\na^2\theta \|_{L^2}^2).
\eeno
For $D_4$,  thanks to integration by parts,  we obtain that
\beno
\int u\cdot\na \na^2 a\, \na^2 a\,dx=-\f12\int \div u |\na^2 a|^2\,dx
\eeno
which implies that
\beno
D_4&\leq& \|\na a\|_{L^3} \|\na^2 u \|_{L^6} \|\na^2 a\|_{L^2}+\|\na u\|_{L^\infty} \|\na^2 a \|_{L^2}^2+\|\div u\|_{L^\infty} \|\na^2 a \|_{L^2}^2\\
&\leq &\eta \|\na^2 a\|_{L^2}^2 +C_\eta ( \|\na a\|_{L^3}^2 \|\na^2 u \|_{L^6}^2+\|\na u\|_{L^\infty}^2 \|\na^2 a \|_{L^2}^2).
\eeno

\medskip

  Combining all the estimates in the above, and using the lower bounds of density and temperature, we obtain that
\begin{equation}\label{aest1}
\begin{split}
&\f12\f{d}{dt} \|\na^2 a\|_{L^2}^2 + C \f1{2\mu+\lam} \|\na^2 a\|_{L^2}^2
\leq C\|\na u\|_{L^\infty}^2\|\na^2 a\|_{L^2}^2\\&\qquad+ C\Big( \|\na^2 G \|_{L^2}^2+\|\na^2 \theta\|_{L^2}^2 +\|\na a\|_{L^3}^2 \|\na^2\theta \|_{L^2}^2+\|\na a\|_{L^3}^2 \|\na^2 u\|_{L^6}^2 \Big).
\end{split}
\end{equation}
 Thanks to \eqref{eq:BD}, we have
\ben\label{uh3}
\begin{split}
&\| \na^2 G  \|_{ L^2}+\|\na^2\curl u\|_{L^2}  \leq \|\na(\rho \dot{u})\|_{ L^2}\leq C ( \|\na a \dot{u}\|_{ L^2}+\| \na \dot{u}\|_{ L^2}) \\
&\leq C ( \|\na a \|_{L^3}\|  \dot{u}\|_{ L^6}+\| \na \dot{u}\|_{ L^2})
  \leq C \| \na \dot{u}\|_{ L^2},\end{split}\een
from which together with $\eqref{controlofu1}_2$ and Proposition \ref{energy_prop_main} imply that
$$\int_0^\infty \Big( \|\na^2 G \|_{L^2}^2+\|\na^2 \theta\|_{L^2}^2 +\|\na a\|_{L^3}^2 \|\na^2\theta \|_{L^2}^2+\|\na a\|_{L^3}^2 \|\na^2 u\|_{L^6}^2 \Big) \,dt\le C.$$

By Gronwall's inequality, we have $ a\in L^\infty((0,+\infty); \dot{H}^2)\cap L^2((0,+\infty); \dot{H}^2)$, from which together with \eqref{uh3}, we deduce that $\na u\in L^2((0,\infty);H^2)$.  We ends the proof of Theorem \ref{thm:main1}.
\end{proof}

\setcounter{equation}{0}
  \section{Global-in-time Stability for system \eqref{cns}}
In this section,   we want to prove Theorem \ref{thm:main2}. The main idea of proof falls into two steps:
\begin{enumerate}
	\item By the local well-posedness for the system \eqref{cns},  we show that the perturbed solution will remain close to the reference solution for a long time if initially they are closer.
	\item With the convergence result implies that the reference solution is close to the  equilibrium after a long time, we can find a time $t_0$ such that $t_0$ is far away from the initial time, and at this moment the solution is close to the equilibrium. Then it is not difficult to prove the global existence in the perturbation framework.
	\end{enumerate}

 \subsection{Setup of the problem}
Let $(\bar{\rho}, \bar{u}, \bar T)$ be a global smooth solution for the \eqref{cns} with the initial data $(\bar{\rho_0},\bar{u_0},\bar{T_0}).$ And let $(\rho, u, T)$ be the solution for the \eqref{cns} associated the initial data $(\rho_0,u_0, T_0),$ which satisfies \eqref{assm1}.
We denote
  $h=\rho-\bar\rho$, $v=u-\bar u$ and $\cT=T-\bar T$ which satisfy that error equations as follows
  \beno
   (ERR)\left\{\begin{array}{l}
 \partial_t h+ (\bar u+v)\cdot \na h=-(h+\bar \rho)\div v-h\div \bar u,\\
\pa_t v -\f1\rho (\mu\Delta v+(\lambda+\mu)\na \div v)   =\f{W_1}{\rho},\\
\pa_t \cT-\f1\rho \Delta \cT=\f{W_2}{\rho}
\end{array}
\right.
  \eeno
  where
  \beno
 && -W_1=h\pa_t \bar u+h \bar u\cdot\na \bar u+\rho\bar u\cdot \na v+\na (h\cT+h\bar T+\bar \rho \cT)+\rho v\cdot\na \bar u+v\cdot \na v   , \\\
  && -W_2=h\pa_t \bar T+\rho\bar u\cdot \na \cT+ h\bar u\cdot \na T+ \rho T\div v+\rho\cT\div u+ h\cT\div \bar u+v\cdot \na \cT \\
  &&\qquad-  \f{\rho\mu}2|\na v+(\na v)'|^2-\rho\lambda(\div v)^2- \rho\mu(\na \bar u+(\na \bar u)')(\na v+(\na v)')-2\rho\lambda\div v \div \bar u.
  \eeno

  \bigskip

 Before proving the stability, we  give the estimate to the term  $\f{W_1}{\rho}$ and $\f{W_2}{\rho}$. We have

  \begin{lem}\label{lem:G-F}
  Assume that $s>\f32$. Let $(\bar{\rho},\bar{u}, \bar T)$ be the smooth solution for \eqref{cns} satisfying \eqref{assm}. There exists a $\eps_0$ such that for any $0<\eps\leq \eps_0$, if
  \beno
  \|h, v, \cT\|_{  L^\infty_T(H^s)}\leq \eps^{\f12},
  \eeno
 then there holds
  \beno
  \|\f{W_1}{\rho}, \f{W_2}{\rho} \|_{   H^{s-1} }\leq C_1\big(\|h\|  _{ H^s }+\|  v \|_{   H^s } +\|\cT\|_{   H^s } \big)
  \big(1+\|h\|_{   H^{s} }+\|  v \|_{  H^{s+1}}+\|  \cT \|_{  H^{s+1}} \big) ,
    \eeno
where $C_1$ is a positive constant depending only on $\mu,$ $\lam,$ and $C$ in \eqref{assm}.
  \end{lem}
 \begin{proof}
 The proof is simply deduced by the standard estimate
 $$\|fg\|_{H^s}\leq \|f\|_{L^\infty}\|g\|_{H^s}+\|f\|_{H^s}\|g\|_{L^\infty},$$
 and $L^\infty\subset H^s$ for $s>\f32.$
 \end{proof}

\subsection{Long time existence of $\mbox{(ERR)}$ }
We want to prove that if the initial data of  $\mbox{(ERR)}$ is small, then its associated solution will be still small during a long time interval. More precisely, we have the following proposition.
 \begin{prop}\label{prop:ST}
Let $(\bar \rho, \bar u,\bar T)$ associated with initial data $(\bar \rho_0, \bar u_0,\bar T_0)$ be a global solution of \eqref{cns}  satisfying \eqref{assm}. Given an $\eps>0,$ if the initial data of $\mbox{(ERR)}$ are determined by the following inequality
  \ben\label{assm2}
\|h_0\|  _{ H^s }+\|  v_0 \|_{   H^s } +\|\cT_0\|_{   H^s }  \leq \eps,
  \een
  then there exists a constant $\delta$ independent of $\eps,$ such that for any $t\in [0, \delta |\ln \eps|]$, there holds
    \beno
 \|h(t)\|  _{ H^s }+\|  v(t) \|_{   H^s } +\|\cT(t)\|_{   H^s }   \leq \eps^{\f12}.
     \eeno
 \end{prop}
\begin{proof} We use the continuity argument to prove the desired result. Let $\Gamma$ be the maximum time such that for any $t\in [0,\Gamma],$ there holds
$$ \|h(t)\|  _{ H^s }+\|  v(t) \|_{   H^s } +\|\cT(t)\|_{   H^s }   \leq \eps^{\f12}.$$
The existence of $\Gamma$ can be obtained by the local well-posedness for the system. The the proof of Proposition \ref{prop:ST} is reduced to prove that $\Gamma \geq \delta |\ln \eps|$ where $\delta>0$ is a constant independent of $\eps.$

By energy estimates, we have
\beno
\f{d}{dt} \|h\|_{H^s}^2&\leq& \|(h+\bar \rho)\div v+h\div \bar u\|_{H^s}^2+ (1+ \|\div(\bar u+v)\|_{L^\infty})\|h\|_{H^s}^2+  \|[ \langle D \rangle^s,(\bar u+v)\cdot \na  ] h\|_{L^2}^2\\
&\leq& C(1+\|v\|_{H^{s+1}}^2) \|h\|_{H^s}^2 .
\eeno
For $v, \cT$, we have for any positive $\delta$, it holds
\beno
&&\f{d}{dt} \|v\|_{H^s}^2 -  \langle\f1\rho (\mu\Delta v+(\lambda+\mu)\na \div v) , v \rangle_{H^s}\\
&&\qquad\leq  C_\delta\big(\|h\|  _{ H^s }+\|  v \|_{   H^s } +\|\cT\|_{   H^s } \big) ^2
  \big(1+\|h\|_{   H^{s} }+\|  v \|_{  H^{s+1}}+\|  \cT \|_{  H^{s+1}} \big)^2+\delta\|  v \|^2_{  H^{s+1}} ,\\
  &&\f{d}{dt} \|\cT\|_{H^s}^2 -  \langle\f1\rho \Delta \cT  , \cT \rangle_{H^s}\\
&&\qquad\leq  C_\delta \big(\|h\|  _{ H^s }+\|  v \|_{   H^s } +\|\cT\|_{   H^s } \big)^2
  \big(1+\|h\|_{   H^{s} }+\|  v \|_{  H^{s+1}}+\|  \cT \|_{  H^{s+1}} \big)^2+\delta\| \cT \|^2_{  H^{s+1}}.
\eeno

All we left is that to prove
\beno
&&-  \langle\f1\rho (\mu\Delta v+(\lambda+\mu)\na \div v) , v \rangle_{H^s}\geq \|\na v\|_{H^s}^2-\|v\|_{H^{s}}\|v\|_{H^{s+1}}- \|h\|_{H^s} \|v\|^2_{H^{s+1}},
\eeno
and
\beno
&& -  \langle\f1\rho \Delta \cT  , \cT \rangle_{H^s}\geq c \|\na \cT\|_{H^s}^2-\|\cT\|_{H^{s}}\|\cT\|_{H^{s+1}}- \|h\|_{H^s} \|\cT\|^2_{H^{s+1}}.
\eeno
On one hand, we have
\beno
-  \langle \f1\rho \Delta \cT  , \cT \rangle_{H^s}  &=& \langle (\f 1{\bar \rho}- \f1\rho )\Delta \cT  , \cT \rangle_{H^s} - \langle \f1{\bar \rho} \Delta \cT  , \cT \rangle_{H^s} \geq - \langle \f1{\bar \rho} \Delta \cT  , \cT \rangle_{H^s} -\|\f h{\bar\rho}\Delta \cT\|_{H^{s-1}}\|\cT\|_{H^{s+1}}\\
&\geq&   \langle \f1{\bar \rho} \na \cT  ,\na \cT \rangle_{H^s}+ \langle \na \f1{\bar \rho} \na \cT  ,\na \cT \rangle_{H^s} - \|h\|_{H^s} \|\cT\|^2_{H^{s+1}}\\
&\geq&  c\|\na \cT\|_{H^s}^2 + \langle [\langle D \rangle^s,\f1{\bar \rho} ]\na \cT  ,\langle D \rangle^s \na \cT \rangle_{L^2}+ \langle \na \f1{\bar \rho} \na \cT  ,\na \cT \rangle_{H^s} - \|h\|_{H^s} \|\cT\|^2_{H^{s+1}}\\
&\geq&  c\|\na \cT\|_{H^s}^2-\|\cT\|_{H^{s}}\|\cT\|_{H^{s+1}}- \|h\|_{H^s} \|\cT\|^2_{H^{s+1}}.
\eeno
By the same argument, we have
\beno
&&-  \langle\f1\rho (\mu\Delta v+(\lambda+\mu)\na \div v) , v \rangle_{H^s}\geq \|\na v\|_{H^s}^2-\|v\|_{H^{s}}\|v\|_{H^{s+1}}- \|h\|_{H^s} \|v\|^2_{H^{s+1}}.
\eeno

Combining all the estimates, and according the definition of $\Gamma,$ for any $t\in [0,\Gamma]$ there holds
\beno
\f{d}{dt}( \|h\|_{H^s}^2 +\|v\|_{H^s}^2 +\|\cT\|_{H^s}^2)+c(\|v\|_{H^{s+1}}^2+\|\cT\|_{H^{s+1}}^2)\leq C( \|h\|_{H^s}^2 +\|v\|_{H^s}^2 +\|\cT\|_{H^s}^2),
\eeno
\beno
 \|h(t)\|_{H^s}^2 +\|v(t)\|_{H^s}^2 +\|\cT(t)\|_{H^s}^2 \leq \|h_0\|_{H^s}^2 +\|v_0\|_{H^s}^2 +\|\cT_0\|_{H^s}^2+C\int_0^t( \|h\|_{H^s}^2 +\|v\|_{H^s}^2 +\|\cT\|_{H^s}^2 )\,d\tau,
\eeno
which implies that $\Gamma \geq \delta |\ln \eps|$ for a suitable $\delta$ independent of $\eps.$ Then the proof is completed.

\end{proof}

  \subsection{Proof of Theorem \ref{thm:main2}.}
  Now we are in a position to prove Theorem \ref{thm:main2}.
 \begin{proof}[Proof of Theorem \ref{thm:main2}]
 First, thanks to Theorem \ref{thm:main1}, we can choose $t_0=\f12(1+|\delta\ln\eps|)$ such that
  \beno
   \|(\bar{\rho}-1)(t_0)\|_{H^1}+\|\bar{u}(t_0)\|_{H^1}+ \|(\bar {T}-1)(t_0)\|_{H^1}\lesssim  (1+|\delta\ln\eps|)^{-\f34}.
  \eeno
Recall that $\rho-1=h+(\bar{\rho}-1),$ $u=v+\bar{u},$ and $T-1=\cT+\bar T-1$ then from Proposition \ref{prop:ST},  we derive that
  \beno
   \|(\rho-1)(t_0)\|_{H^1}+\| u(t_0)\|_{H^1}+ \|( T-1)(t_0)\|_{H^1}\lesssim  \eps^{\f12}+(1+|\delta\ln\eps|)^{-\f34}\lesssim (1+|\delta\ln\eps|)^{-\f34}.
  \eeno
This means that at the time $t_0$, the system \eqref{cns} is in the close-to-equilibrium regime. Then thanks to the results in \cite{MN1, MN2}, we obtain the global existence for $(\rho-1, u, T-1).$ Moreover due to the definition of $\Gamma,$ we conclude that
\beno
\|h\|  _{L^\infty_t( H^s) }+\|  v\|_{   L^\infty_t( H^s)} +\|\cT\|_{ L^\infty_t( H^s) }  \lesssim \min\{(1+\delta|\ln\eps|)^{-\f34}, (1+t)^{-\f34}+\epsilon\}.  \eeno
It completes the proof of Theorem \ref{thm:main2}.
  \end{proof}

\section*{Acknowledgement}
  Lingbing He is supported by NSF of China under Grant 11771236. Jingchi Huang is supported by NSF of China under Grant 11701585 and Science and Technology Planning Project of Guangdong under Grant 2017A030310047. Chao Wang is supported by NSF of China under Grant 11701016.

\end{document}